\documentclass[reqno]{amsart}

\usepackage[utf8]{inputenc}
\usepackage[T1]{fontenc}
\usepackage[english,activeacute]{babel}
\usepackage{amssymb}
\usepackage{amsfonts}
\usepackage{amsthm}
\usepackage{calrsfs}
\usepackage{latexsym}
\usepackage{array}
\usepackage{bbm}
\usepackage{stmaryrd}

\usepackage{url}
\usepackage{mathtools}
\usepackage[svgnames]{xcolor}
\usepackage{hyperref}

\newcommand{\m}[1]{\mathbb{#1}}
\newcommand{\gh}[1]{\mathfrak{#1}}
\newcommand{\q}[1]{\mathcal{#1}}

\newcommand{\ds}{\displaystyle}

\newcommand{\e}{\varepsilon}
\newcommand{\ve}{\varepsilon}

\newcommand{\bt}{\beta}

\newcommand{\tendf}{\rightharpoonup}

\providecommand{\abs}[1]{\left\lvert #1 \right\rvert}

\DeclareMathOperator{\Span}{\mathrm{Span}}

\DeclareMathOperator{\Id}{\mathrm{Id}}

\renewcommand{\le}{\leqslant}
\renewcommand{\ge}{\geqslant}

\theoremstyle{plain}
\newtheorem{thm}{Theorem}
\newtheorem*{thm*}{Theorem}
\newtheorem{prop}{Proposition}
\newtheorem{cor}{Corollary}
\newtheorem{lem}{Lemma}

\theoremstyle{definition}
\newtheorem{defi}{Definition}

\theoremstyle{remark}

\newtheorem*{claim}{Claim}

\newcommand{\R}{\mathbb{R}}
\newcommand{\N}{\mathbb{N}}

\newcommand{\la}{\lambda}

\newcommand{\al}{\alpha}
\newcommand{\ga}{\gamma}

\def\bm{\left( \begin{array}{cc}}
\def\endm{\end{array}\right)}

\newcommand{\be}{\begin{equation}}
\newcommand{\ee}{\end{equation}}
\newcommand{\ba}{\begin{equation*}}
\newcommand{\ea}{\begin{equation*}}
\newcommand{\bea}{\begin{eqnarray}}
\newcommand{\eea}{\end{eqnarray}}
\newcommand{\bee}{\begin{eqnarray*}}
\newcommand{\eee}{\end{eqnarray*}}
\newcommand{\ben}{\begin{enumerate}}
\newcommand{\een}{\end{enumerate}}
\newcommand{\nonu}{\nonumber}
\makeatletter
\def\blfootnote{\xdef\@thefnmark{}\@footnotetext}
\makeatother

\title{Multi-solitons for nonlinear Klein-Gordon equations}

\date{}
\author{Rapha\"el C\^ote and Claudio Mu\~noz}



\thanks{The first author wishes to thank the University of Chicago for its hospitality during the academic year 2011-12, and acknowledges support from the European Research Council through the project BLOWDISOL}

\begin{document} 

\begin{abstract}
In this paper we consider the existence of multi-soliton structures for the nonlinear Klein-Gordon equation \eqref{NLKG} in $\R^{1+d}$. We prove that, independently of the unstable character of \eqref{NLKG} solitons, it is possible to construct a $N$-soliton family of solutions to \eqref{NLKG}, of dimension $2N$, globally well-defined in the energy space $H^1\times L^2$ for all large positive times. The method of proof involves the generalization of previous works on supercritical NLS and gKdV equations by Martel, Merle and the first author \cite{CMM} to the wave case, where we replace the unstable mode associated to the linear NLKG operator by two generalized directions that are controlled without appealing to modulation theory. As a byproduct, we generalize the linear theory described in Grillakis-Shatah-Strauss \cite{GSS1} and Duyckaerts-Merle \cite{DM2} to the case of boosted solitons, and provide new solutions to be studied using the recent Nakanishi-Schlag \cite{NS1} theory.  
\end{abstract}

\maketitle \markboth{NLKG multi-solitons} {Rapha\"el C\^ote and Claudio Mu\~noz}
\renewcommand{\sectionmark}[1]{}

\allowdisplaybreaks

\section{Introduction}

\medskip

In this paper we are interested in the problem of constructing multi-soliton solutions for some well-known scalar field equations. Let $f=f(s)$  be a real-valued $\q C^1$-function. We consider the nonlinear Klein-Gordon equation (NLKG) in $\R^{1+d}$, $d\geq 1$,
\begin{equation}
\label{NLKG}\tag{NLKG}
\partial_{tt} u - \Delta u + u -f(u)  = 0, \quad u(t,x) \in \R, \quad (t,x)\in \R \times \R^d. 
\end{equation}
This equation arises in Quantum Field Physics as a model for a self-interacting, nonlinear scalar field, invariant under Lorentz transformations (see below). 

\medskip

Let $F$ be the standard integral of $f$:
\be \label{def:F}
\ds F(s) := \int_0^s f(\sigma)d\sigma.
\ee
We will assume that for some fixed constant $C>0$,
\ben
\item[(A)] If $d=1$,
\ben
\item[$(i)$] $f$ is odd, and $f(0) = f'(0) =0$, and
\item[$(ii)$] There exists $s_0>0$ such that $\ds F(s_0) - \frac 12 s_0^2 >0$.
\een
\item[(B)] If $d \ge 2$, $f$ is a pure power $H^1$-subcritical nonlinearity: $f(u) = \la |u|^{p-1} u$, where $\la>0$, $p\in (1,1+\frac 4{d-2})$.
\een

Prescribing $f$ to the above class of focusing nonlinearities ensures that the corresponding Cauchy problem for \eqref{NLKG} is locally well-posed in $H^s(\R^d) \times H^{s-1}(\R^d)$, for any $s\ge 1$: we refer to Ginibre-Velo \cite{GV85} and Nakamura-Ozawa \cite{NO01} (when $d=2$) for more details.  


\medskip

Also under the above conditions, the Energy and Momentum (every integral is taken over $\R^d$)
\begin{align}
E[u,u_t] (t) & =  \frac{1}{2}  \int \left[ |\partial_t u(t,x)|^2+ | \nabla u(t,x)|^2 + |u(t,x)|^2 - 2 F(u(t,x))  \right] dx, \label{E} \\
P[u,u_t] (t) & =  \frac{1}{2} \int \partial_t u(t,x)  \nabla u(t,x) \ dx,\label{P}
\end{align}
are conserved along the flow.

\medskip

Another important feature of equation \eqref{NLKG}, still under the previous conditions, is the fact that it admits stationary solutions of the form $u(t,x) = U(x)$ (i.e., with no dependence on $t$). Among them, we are interested in the ground-state $Q=Q(x)$, where $Q$ is a positive solution of the elliptic PDE
\be\label{Q}
\Delta Q - Q  + f(Q) = 0, \quad Q>0, \quad Q\in H^1(\R^d).
\ee
The existence of this solution goes back to Berestycki-Lions \cite{BL},  provided the above conditions (in particular $(ii)$)  hold. Additionally, it is well-known that $Q$ is radial and exponentially decreasing, along with its first and second derivatives (Gidas-Ni-Nirenberg \cite{GNN}), and \emph{unique} up to definition of the origin (see Kwong \cite{K}, Serrin and Tang \cite{ST}). 

\medskip

In fact, our main result written below could be extended to more general nonlinearity under an additional assumption of spectral nature, namely that the linearized operator around $Q$ has a standard simple spectrum. More precisely, Theorem \ref{MT1} holds, as soon as $f$ satisfies $(i)$, $(ii)$ and:
\ben
\item[$(iii)$] If $d=2$, $|f'(s)| \le C |s|^p e^{\kappa s^2}$, for some $p \ge 0$, $\kappa >0$ and all $s \in \m R$.
\item[$(iv)$] If $d \ge 3$, $|f'(s)| \le C (1+|s|^{p-1})$ for some $\ds p < 1+ \frac{4}{d-2}$ and all $s \in \m R$.
\item[$(v)$] $-\Delta z + z  - f'(Q)z $ has a unique simple negative eigenvalue, and its kernel is given by $\{ x \cdot \nabla Q | x \in \m R^d \}$ and it is nondegenerate.
\een
Assumption $(v)$ has been checked in the cases (A) and (B) (using ODE analysis), and is believed to hold for a wide class of functions $f$. (See Lemma \ref{lemLd}.)

\medskip

Since \eqref{NLKG} is invariant under \emph{Lorentz boosts}, we can define a \emph{boosted ground state} (a \emph{soliton} from now on) with relative velocity $\beta \in \m R^d$.  More precisely, let $\beta = (\beta_1, \dots, \beta_d) \in \R^d$, with $|\beta|<1$ (we denote $|\cdot|$ the euclidian norm on $\m R^d$), the corresponding Lorentz boost is given by the $(d+1) \times (d+1)$ matrix 
\begin{equation}
\label{Lor}
\Lambda_\bt := \begin{pmatrix}
\gamma & -\beta_1 \gamma & \cdots & \beta_d \gamma \\
- \beta_1 \gamma &  \\
\vdots  & & \ds \Id_d + \frac{(\gamma-1)}{|\beta|^2} \beta \beta^T \\
-\beta_d \gamma
\end{pmatrix} \quad \text{where} \quad \ga := \frac{1}{\sqrt{1-|\bt|^2}},
\end{equation}
($\beta \beta^T$ is the $d \times d$ rank 1 matrix with coefficient of index $(i,j)$ $\beta_i \beta_j$).
 Then the boosted soliton with velocity $\bt$ is
\begin{equation}
\label{Qb}
Q_\bt(t,x) := Q \left( \Lambda_\beta \begin{pmatrix} t \\ x \end{pmatrix} \right),
\end{equation}
where with a slight abuse of notation we say that $Q(t,x) = Q(x)$ (namely we project on the last $d$ coordinates). Also notice that \eqref{NLKG} is invariant by space translation (shifts).  
Hence the general family of solitons is parametrized by speed $\beta \in \m R^d$ and shift (translation) $x_0 \in \m R^d$:
\[ (t,x) \mapsto Q_{\beta}(t,x-x_0). \]
This family is the orbit of $\{ Q \}$ under all the symmetries of \eqref{NLKG} (general Lorentz transformation, time and space shifts), in particular it is invariant under these transformations: see the Appendix \ref{Ap} for further details.

\medskip

In the rest of this work, it will be convenient to work with vector data $(u, \partial_t u)^{T}$. For notational purposes, we use upper-case letters to denote vector valued functions and lower-case letters for scalar functions (except for the scalar field $Q_\beta$).

\medskip

We will work in the energy space $H^1(\R^d) \times L^2(\R^d)$ endowed with the following scalar product: denote $U =
(u_1,u_2)^T$, $V = (v_1,v_2)^T$, we define
\begin{align}
\label{nor}
\langle U | V \rangle = \left\langle \begin{pmatrix} u_1 \\ u_2 \end{pmatrix} \middle| \begin{pmatrix} v_1 \\ v_2 \end{pmatrix} \right\rangle := ( u_1 | v_1 ) + (u_2 | v_2 ) = \int (u_1 v_1 + u_2 v_2), \\
\text{where} \quad (u|v) := \int u v, \nonumber
\end{align}
and the energy norm
\be\label{Nor}
\|U\|^2 := \langle U | U \rangle + (\nabla u_1| \nabla u_1) =  \| u_1 \|_{H^1}^2 + \| u_2 \|_{L^2}^2.
\ee
It is well known (see e.g. Grillakis-Shatah-Strauss \cite{GSS1}) that $(Q,0)$ is unstable\footnote{This result was known in the physics
  literature as the Derrick's Theorem \cite{D}.} in the energy space.
The instability properties of $Q$ and solution with energy slightly above $E[(Q,0)]$ have recently been further studied by Nakanishi and Schlag, see \cite{NS1} and subsequent works. Their ideas are further developments of the primary idea introduced in Duyckaerts-Merle \cite{DM1}, in the context of the energy critical nonlinear wave equation (where the relevant nonlinear object is the stationary function $W$ which solves $\Delta W+W^{1+4/(d-2)}=0$).

\medskip

In this paper, we want to understand the dynamics of large, quantized energy solutions. More precisely, we address the question whether is it possible to construct a multi-soliton solution for (\ref{NLKG}),
i.e. a solution $u$ to \eqref{NLKG} defined on a semi-infinite interval of time, such that
\[ u(t,x) \sim \sum_{j=1}^N Q_{\beta_j}(t,x-x_j) \quad \text{as} \quad t \to + \infty. \]
Such solutions were constructed for the nonlinear Schrödinger equation and the generalized Korteweg-de Vries equation, first in the $L^2$-critical and subcritical case by Merle \cite{Merle1}, Martel \cite{Martel1} and Martel-Merle \cite{MM06}. These results followed from the stability and asymptotic stability theory that these authors developed.

\medskip

The existence of multi-solitons was then extended by Martel-Merle and the first author \cite{CMM} to the $L^2$ supercritical case: in this latter case, each single soliton is unstable, hence the multi-soliton is a highly unstable solution. It turns out that this is also the case for scalar field equations as \eqref{NLKG}. We prove that, regardless of the instability of the soliton, one can construct large mass multi-solitons, on the whole range of parameters $\beta_1, \dots, \beta_N \in \m R^d$ distinct, with $|\beta_j|<1$ and $x_1, \dots, x_N \in \m R^d$. More precisely, the main result of this paper is the following.

\begin{thm}\label{MT1}
Assume (A) or (B), and let $\bt_1, \bt_2 ,\ldots ,\bt_N \in \m R^d$ be a set of different velocities
\[ \forall i \ne j, \quad \beta_i \ne \beta_j, \quad \text{and} \quad |\beta_j| <1, \]
 and $x_1, x_2, \ldots, x_N \in \m R^d$ shift parameters.  

Then there exist a time $T_0 \in \R$, constants $C>0$, and $\ga_0>0$, only depending on the sets $(\beta_j)_j$, $(x_j)_j$, and a solution $(u, \partial_t u) \in \q C([T_0,+\infty), H^1(\R^d)\times L^2(\R^d))$ of (\ref{NLKG}), globally defined for forward times and satisfying
\[
\forall t \ge T_0, \quad  \Big\| (u,\partial_t u)(t,x) - \sum_{j=1}^N (Q_{\bt_j} , \partial_t Q_{\bt_j})(t,x-x_j)\Big\| \le C e^{-\ga_0 t}.
\]
\end{thm}

We remark that this is the first multi-soliton result for wave-type equations. 
Although the nonlinear object under consideration
is the same as for (NLS) for example, the structure of the flow is
different (recall that all solitons are unstable for \eqref{NLKG},
irrespective of the nonlinearity). Hence we need to work in a more general
framework, the one given by a matrix description of \eqref{NLKG}. 

\medskip

Let us describe the main steps of the proof. We first revisit the standard spectral theory of linearized operators around the soliton, and the second order derivative of the energy-momentum functional (see $H$ in \eqref{H})  \cite{GSS1}. Since solitons are unstable objects, it is clear that such a theory will not be enough to describe the dynamics of several solitons. However, a slight variation of this functional (see $\mathcal H$ in \eqref{qH}) turns out to be the key element to study. We describe its spectrum in great detail, in particular we prove that this operator has three eigenvalues: the kernel zero, and two opposite sign eigenvalues, with associated eigenfunctions $Z_\pm$. After some work we are able to prove a coercivity property for the operator $\mathcal H$ modulo the two directions $Z_+$ and $Z_-$. This analysis was first conducted by Pego and Weinstein \cite{PW} in the context of generalized KdV equations. 

\medskip

The rest of the work is devoted to the study of the dynamics of small perturbations of the sum of $N$ solitons, in particular how the two directions associated to $Z_\pm$ evolve. Using a topological argument, we can show the existence of suitable initial data for (NLKG) such that both directions remain controlled for all large positive time, proving the main theorem.   We remark that this method is general and does not require the study of the linear evolution at large, but also a deep understanding of suitable alternative directions of the linearized operator. A nice open question should be the extension of this result to the nonlinear wave case, where the soliton decays polynomially.

\medskip

For the sake of easiness and clarity, we present the detailed computations in the one dimensional case $d=1$. This case encompass all difficulties, the higher dimension case adding only indices and notational inconvenience: we will briefly describe the corresponding differences at the end of each section.

\subsection*{Organization of this paper} In Section 2, we develop spectral aspects of the linearized flow around $Q_\beta$, which are more subtle than in the (NLS) or (gKdV) case. In Section 3, we construct approximate $N$-soliton solutions in Proposition \ref{prop:app_Nsol}, which we do by estimation backward in time as in \cite{Merle1,Martel1,MM06}. There we present the nonlinear argument, relying \emph{in fine} on a topological argument as in \cite{CMM}. The Lyapunov functional has to be chosen carefully, as we cannot allow mixed derivatives of the form $\partial_{tx} u$.
Finally in Section \ref{sec:4}, we prove Theorem \ref{MT1}, relying on the previously proved Proposition \ref{prop:app_Nsol} and a compactness procedure.

\begin{center}
\bfseries Acknowlegment
\end{center}

We would like to thank Wilhelm Schlag for pointing us this problem and for enlightening discussions.
We are deeply indebted to the anonymous referee, who we thank for his thorough reading and comments which improved the manuscript significantly.

\bigskip

\section{Spectral theory}

\medskip

In this section we describe and solve two spectral problems related to (\ref{NLKG}). We will work with functions independent of time, unless specified explicitly. The main result of this section is Proposition \ref{prop:22}.

\subsection{Coercivity of the Hessian}

First of all, we recall the structure of the Hessian of the energy around $Q$. Given $Q=Q(x)$ ground state of (\ref{Q}) and $Q_\bt(x) = Q(\ga x)$, where  $\ga = (1-\bt^2)^{-1/2}$, we define the operators
\be\label{Lp}
L^+ := -\partial_{xx} + \Id - f'(Q), \quad \text{and} \quad L_\beta^+ = -\ga^{-2} \partial_{xx} + \Id - f'(Q_\beta),
\ee
so that $L_\beta^+$ is a rescaled version of $L_+$: 
\[ L_{\beta}^+ \left( v \left( \ga x 
\right) \right) = (L^+ v) \left( \ga x 
\right). \]
As a consequence of the Sturm-Liouville theory and the previous identity, we have the following spectral properties for $L^+$, and therefore for $L^+_\beta$.

\begin{lem} \label{lemL}
The unbounded operator $L^+$, defined in $L^2(\R)$ with domain $H^2(\R)$, is self-adjoint, has a unique negative eigenvalue $-\la_0<0$ (with corresponding $L^2$-normalized eigenfunction $Q^-$) and its kernel is spanned by $\partial_x Q$. Moreover, the continuous spectrum is $[1, +\infty)$, and $0$ is an isolated eigenvalue.
\end{lem}

We recall that from  standard elliptic theory, $Q^-$ is smooth, even and exponentially decreasing in space: there exists $c_0 >0$ such that
\be\label{Qminus}
\forall \ k \in \N, \ \forall \ x \in \R, \ \exists \ C_k, \quad |\partial_x^k Q^-(x)|  \le C_k e^{-c_0|x|}. 
\ee
It is not difficult to check that one can take any $c_0$ satisfying $0<c_0 \le \sqrt{1+\la_0}$.

\medskip

Another consequence of Lemma \ref{lemL} is the following fact: $L^+_\beta$ has a unique negative eigenvalue $- \la_0$ with (even) eigenfunction $Q^-_\beta(x) : = Q^-( \ga x)$, its kernel is spanned by $\partial_x Q_\beta$ and has continuous spectrum $[1,+\infty)$. Additionally, we have

\begin{cor} \label{coerL}  
There exists $\nu_0 \in (0,1)$ such that, if  $v\in H^1(\R)$  satisfies $(v| Q^-_\beta) = (v|\partial_x Q_\beta) =0$, then  $( L^+_\beta v | v ) \ge  \nu_0\| v \|^2_{H^1}$.
\end{cor}

We introduce now suitable matrix operators associated to the dynamics around a soliton. These operators will be dependent on the velocity parameter $\beta$, but for simplicity of notation, we will omit the subscript $\beta$ when there is no ambiguity. Define\footnote{Do not confuse with the transpose symbol $(\cdot )^T$.}
\be\label{T}
T=T_\beta := - \partial_{xx} + \Id - f'(Q_\bt) = L^+_\beta - \beta^2 \partial_{xx}, 
\ee
\be\label{JJ}
J := \begin{pmatrix}
0 & 1 \\
-1 & 0
\end{pmatrix}, 
\ee
\be\label{LL}
L := \begin{pmatrix}
T & 0 \\
0 & \Id
\end{pmatrix}, 
\ee
and
\begin{align} \label{H}
H &: = L -  J \begin{pmatrix}
 \beta \partial_{x} & 0  \\
0 & \beta \partial_{x}
\end{pmatrix} =  \begin{pmatrix}
T & -\beta \partial_{x} \\
\beta \partial_{x} & \Id
\end{pmatrix}.
\end{align}
The operator $H$ is the standard second order derivative of the functional for which the vector soliton $R = (Q_{\bt} , \partial_t Q_{\bt})^T$ is an associated local minimizer. Later we will discuss in detail this assertion. The following Proposition describes the main spectral properties of $H$. Recall that $\langle \cdot | \cdot \rangle$ and $( \cdot | \cdot )$ denote the symmetric bilinear forms on $H^1(\R) \times L^2(\R)$ and $L^2(\R) $ respectively, introduced in (\ref{nor}), and $\|\cdot\|$ is the energy norm defined in \eqref{Nor}.

\begin{prop}\label{PropH}
Let $\bt\in \R$, $|\bt|<1$. The matrix operator $H$, defined in $L^2(\R)\times L^2(\R)$ with dense domain $H^2(\R)\times H^1(\R)$, is a self-adjoint operator. Furthermore, there exist $\alpha_0 >0$, $\Phi_0$, $\Phi_{-} \in  \q S(\m R)^2$ (with exponential decay, along with their derivatives) such that
\be\label{10}
 H\Phi_0 = 0, \quad \langle \Phi_0 | \Phi_- \rangle =0,  
\ee
\be\label{11}
 \langle H \Phi_{-} | \Phi_{-} \rangle < 0,
\ee
and the following coercivity property holds. Let $V = (v_1,v_2)^{T} \in H^1(\R)\times L^2(\R)$. Then,
\be \label{coer1}
\text{if} \quad  \langle V|\Phi_{0} \rangle= \langle V| H \Phi_- \rangle =0 \quad \text{one has} \quad  \langle HV |V\rangle \geq \al_0 \|V\|^2.
\ee
\end{prop}

A stronger version of this result was stated by Grillakis, Shatah and Strauss
in \cite[Lemma 6.2]{GSS1}, but the proof given there contained a gap, as noted
in the errata at the end of \cite[page 347]{GSS2}. As a replacement, the
Proposition above (weaker than the original Grillakis-Shatah-Strauss result, but adequate for our purposes) was proposed in
the errata \cite{GSS2}, without proof. We have not found a clear definition and meaning of the function 
$\Phi_{-}$ in \cite{GSS2}, so therefore, for the convenience of the reader, we write the details of the
proof in the following lines.

\begin{proof}[Proof of Proposition \ref{PropH}] 
It is easy to check that $H$ is indeed a self-adjoint operator. On the
other hand, let $V = (v_1, v_2)^{T}$. We have from (\ref{H}), 
\begin{align}\label{Comp1}
\langle H V | V \rangle & =
\left\langle \begin{pmatrix} 
Tv_1 - \beta \partial_x v_2 \\
\beta \partial_xv_1 + v_2 
\end{pmatrix} \middle| \begin{pmatrix} v_1 \\ v_2 \end{pmatrix} \right\rangle \nonu\\
& =  (Tv_1| v_1) -  \beta (\partial_x v_2 | v_1) + \beta (\partial_x v_1 | v_2) + (v_2|v_2) \nonumber \\
& = (L_\beta^+ v_1 | v_1) + \beta^2  ( \partial_x v_1 | \partial_x v_1)  + 2 \beta ( v_2 | \partial_x v_1) +(v_2|v_2) \nonumber \\
&= (L_\beta^+ v_1| v_1) + ( \beta \partial_x v_1 + v_2 | \beta \partial_x v_1 + v_2).
\end{align}
Recalling the notation of Corollary \ref{coerL}, we define 
\be\label{Phi0}
 \Phi_0 := \begin{pmatrix}
\partial_x Q_\beta \\
- \beta \partial_{xx} Q_\beta
\end{pmatrix}, \quad 
 \Phi_- := \begin{pmatrix}
Q^-_\beta \\
-\bt \partial_x Q^-_\beta
\end{pmatrix}, 
\ee
One can check from (\ref{Comp1}) that $\langle \Phi_0 | \Phi_0 \rangle \neq 0$ and $\langle H \Phi_0 | \Phi_0 \rangle =0$, since $L_{\beta}^+ \partial_x Q_\beta =0$. Note additionally that by parity $\langle  \Phi_- |  \Phi_0 \rangle =0.$ Therefore, (\ref{10}) is directly satisfied. Also notice that 
\be\label{Hfi}
H\Phi_- = -\la_0 \begin{pmatrix}
Q_\beta^- \\
0
\end{pmatrix}.
\ee

\medskip

We now prove \eqref{coer1}. Let $V = (v_1,v_2)^{T} \in H^1(\R)\times L^2(\R)$ be satisfying the orthogonality properties
\[
\langle V|\Phi_{0} \rangle= \langle V| H\Phi_- \rangle =0. 
\]
Let us decompose $v_1$ in terms of the nonpositive spectral elements of $L_\bt^+$, and $L^2$-orthogonally:
\[
v_1 = a Q^-_\beta + b \partial_x Q_\beta + q,  \qquad (q|Q^-_\beta)=(q|\partial_x Q_\beta)=0. 
\]
From the orthogonality conditions in (\ref{coer1}), we have 
\[
\left\langle  \begin{pmatrix}
v_1 \\
v_2
\end{pmatrix}  \Big|  \begin{pmatrix}
Q_\beta^- \\
0
\end{pmatrix} \right\rangle =0,
\]
so that $a=0$, and hence from Corollary \ref{coerL},
\be\label{Hpos} 
\langle H V | V \rangle = (L_\beta^+ q|q) + ( \beta \partial_x v_1 + v_2| \beta \partial_x v_1 + v_2) \ge \nu_0 \| q \|_{H^1}^2 \ge 0.
\ee
We now argue by contradiction. Assume that there exists a normalized sequence $V^n = (v_1^n,v_2^n)^{T} \in H^1(\R)\times L^2(\R)$ that satisfies the orthogonality properties
\be\label{hypo1}
\langle V^n|\Phi_{0} \rangle= \langle V^n| H\Phi_- \rangle =0, \quad \| V^n \|^2 =1, \quad \text{and such that} \quad \langle H V^n | V^n \rangle \to 0.
\ee
Let us write the $L^2$-orthogonal decomposition for each $v_1^n$:
\[ 
v_1^n = b_n \partial_x Q_\beta + q_n, \qquad (q_n |\partial_x Q_\beta)=0.
\]
Then in view of \eqref{Hpos} and \eqref{hypo1} applied this time to the sequence $V^n$, $q_n \to 0$ in $H^1$ and $ \beta \partial_x v_1^n + v_2^n \to 0$ in $L^2$.  Now we compute
\begin{align*}
0 & = \langle V^n | \Phi_0 \rangle = \int (v_1^n \partial_x Q_\beta - v_2^n \beta \partial_{xx} Q_\beta) \\
& = \int v_1^n \partial_x Q_\beta +  \beta\int (\beta \partial_x v_1^n + o_{L^2} (1)) \partial_{xx} Q_\beta \\
& = b_n \| \partial_x Q_{\beta} \|_{L^2}^2  +\beta^2 \int  (  b_n \partial_{xx} Q_\beta +  \partial_x q_n) \partial_{xx} Q_\beta + o(1) \\
& = b_n ( \| \partial_x Q_{\beta} \|_{L^2}^2 + \beta^2 \| \partial_{xx} Q_\beta \|_{L^2}^2) - \beta^2 \int q_n \partial_{xxx} Q_\beta +o(1)
\end{align*}
Now $q_n \to 0$ in $L^2$, so that $( q_n | \partial_{xxx} Q_\beta ) \to 0$, and hence $b_n \to 0$ as $n \to +\infty$. But in this case, $v_1^n = b_n \partial_x Q_\beta + q_n \to 0$ in $H^1$ and $v_2^n = \beta \partial_x v_1^n  + o_{L^2}(1) \to 0$ in $L^2$. Hence $\| V^n \|^2 = \| v_1^n \|_{H^1}^2 + \| v_2^n \|_{L^2}^2 \to 0$, a contradiction to (\ref{hypo1}). 

%
%

\medskip

It remains to show that $\langle H \Phi_{-} | \Phi_{-} \rangle <0$, namely (\ref{11}). Indeed,
\[ 
\langle H \Phi_{-} | \Phi_{-} \rangle= -\la_0
\left\langle  \begin{pmatrix}
Q^-_\beta \\
0
\end{pmatrix} \middle| \begin{pmatrix}
Q^-_\beta \\
\partial_x Q^{-}_\beta
\end{pmatrix} \right\rangle= - \la_0 \| Q^-_\beta \|_{L^2}^2 < 0. \]
\end{proof}

\subsection{Eigenfunctions of the linearized flow and Hessian}

It is still unclear whether or not the coercivity property \eqref{coer1} -- a key point in
the proof of any stability result -- is useful for us, since solitons are actually unstable. It turns out that for our purposes, we need a different version of Proposition \ref{PropH}, for the linearized operator of the {\bf flow around $Q$}. In order to state such a result, we introduce some additional notation. 

\medskip

Let $\bt\in \R$, $|\bt|<1$ be a Lorentz parameter, and consider the operators $T$, $J$, $L$ and $H$ defined in (\ref{T})-(\ref{H}). Let
\be\label{qL}
\q L = \q L(\bt) = JL = \begin{pmatrix}
0 & \Id \\
-T& 0
\end{pmatrix} , 
\ee
and
\be\label{qH}
\q H  
= \begin{pmatrix}
- \beta \partial_{x} & -T \\
\Id & - \beta \partial_{x}
\end{pmatrix} = -HJ .
\ee

Concerning this last operator, we prove the following result.

\begin{lem}\label{Z}
Let $\bt\in \R$, $|\bt|<1$, $\ga = (1-\bt^2)^{-1/2}$ and $\la_0$ from Lemma \ref{lemL}. There are functions $Z_{0} =Z_{0,\bt}$, and $Z_{\pm} = Z_{\pm,\bt}$, with components exponentially decreasing in space, satisfying the spectral equations
\be\label{Zpm}
\q H Z_{0} = 0, \qquad \hbox{ and } \qquad
\q H Z_{\pm} = \pm \frac{\sqrt{\la_0}}{\gamma} Z_\pm. 
\ee
Moreover, by the nondegeneracy of the kernel spanned by $\Phi_{0}$, we can assume $\Phi_{0} = JZ_{0}$.
\end{lem}

\begin{proof}
The proof is similar to that of \cite{GSS1}. In particular, we obtain explicit expressions for $Z_0$ and $Z_\pm$ in the following lines. 

\medskip

The eigenvalue problem $\q H Z =\la Z$ reads now, with $Z(x) = (Z_1(\ga x), Z_2(\ga x))^T$, 
\be\label{Z1Z2}
T Z_2 +  \beta (Z_1)_x + \la Z_1 =0, \quad Z_1 - \beta (Z_2)_x -\la Z_2 =0, 
\ee
Replacing $Z_1$ in the first equation above, we get in the variable $s=\ga x$ (recall that $Q_\bt (x) = Q(\ga x)$),
\[ 
-\ga^2 Z_2''  + Z_2 - f'(Q) Z_2 +\bt \ga (\la Z_2' +\bt \ga  Z_2'') + \la (\bt\ga Z_2' + \la Z_2) =0, 
\]
namely
\be\label{Z2}
-Z_2'' + Z_2 - f'(Q) Z_2 + 2\bt \ga \la Z_2'  =   -\la^2 Z_2.
\ee
Performing the transformation $Z_2(s) := \tilde Z_2(s) e^{ \bt \ga \la s} $, where $s\in \R$, we get
\[
- \tilde Z_2'' + \tilde Z_2 - f'(Q)\tilde Z_2  = -( \bt^2 \ga^2 +1)\la^2 \tilde Z_2 = -\la^2 \ga^2 \tilde Z_2.
\]
Therefore, by virtue of  Lemma \ref{lemL} we can take $\tilde Z_2 = Q^-(s)$ and $\la_\pm \ga = \pm \sqrt{\la_0}$, where $-\la_0<0$ 
 is the first eigenvalue of the standard Schr\"odinger operator $L^+$, defined in (\ref{Lp}).  Thus, 
\[
Z_{\pm,2}(s) =Q^-(s) e^{\pm \bt \sqrt{\la_0} s}.
\]
Note that from (\ref{Qminus}), $Z_{\pm,2}$ decreases exponentially at both sides of the origin, since $|\bt|<1$ and $\bt \sqrt{\la_0} -\sqrt{1+\la_0}<0$ .

\medskip

From (\ref{Z1Z2}), we have
\begin{align*}
Z_{\pm,1} (s) & =  \bt \ga Z_{\pm,2}'(s) + \la_\pm Z_{\pm,2}(s)  \\
& =  \Big[\bt \ga (Q^-)_s  \pm \bt^2 \ga \sqrt{\la_0}   Q^-   \pm \frac{\sqrt{\la_0}}{\ga} Q^- \Big] e^{\pm \bt \sqrt{\la_0} s} \nonu\\
& = \ga \big[\bt  (Q^-)_s  \pm  \sqrt{\la_0}   Q^-   \big] e^{\pm \bt \sqrt{\la_0} s}.
\end{align*}
By the same reasons as above, $Z_{\pm,1}$ is an exponentially decreasing function. From these identities, we have
\begin{align}\label{decr1}
Z_{\pm} (x) &= \left( \begin{array}{c}  \ga \bt  (Q^-)_s (\ga x) \pm  \ga \sqrt{\la_0}   Q^- (\ga x)    \\   Q^-(\ga x) \end{array}\right)e^{\pm \bt \sqrt{\la_0} \ga x} \nonu \\
& =  \left( \begin{array}{c}   \bt  (Q^-_\bt)_x (x) \pm  \ga \sqrt{\la_0}   Q^-_\bt (x)   \\   Q^-_\bt (x) \end{array}\right)e^{\pm \bt \sqrt{\la_0} \ga x}.
\end{align}
Now, we consider the computation of $Z_0$. Replacing $\la=0$ in (\ref{Z2}), we can choose
\[
Z_{0,2}(s) =Q'(s), \quad  \hbox{ and } \quad Z_{0,1} = \bt \ga Q''(s), 
\]
from which we get
\be\label{decr2}
Z_0 (x)= \ga \left( \begin{array}{c}  Z_{0,1}(x)   \\   Z_{0,2} (x)\end{array}\right) = \ga \left( \begin{array}{c}  \bt \ga Q''(\ga x)   \\   Q' (\ga x)\end{array}\right) = \left( \begin{array}{c}  \bt Q_\bt''( x)   \\   Q_\bt ' (x)\end{array}\right) .
\ee
It is clear that $\q H  Z_0= 0 $. Similarly, we have $\ds \q H  Z_{\pm}=  \pm \frac{\sqrt{\la_0}}{\ga} Z_\pm $, which proves (\ref{Zpm}).
\end{proof}

In order to prove Proposition \ref{prop:22}, we need to prove the existence of two additional functions, both associated to $Z_\pm$.

\begin{lem}\label{Ypm}
There exist unique functions $Y_\pm$, with components exponentially decreasing in space, such that 
\[
 H Y_\pm = Z_\pm, \qquad \langle \Phi_{0} | Y_\pm \rangle=0.
\]
Moreover, $Y_\pm$ satisfy the additional orthogonality conditions $\langle Y_\pm | HY_\pm \rangle=0$. 
\end{lem}

\begin{proof}
Let us prove the existence of $Y_\pm$. It is well-known that a necessary and sufficient condition for existence is the following  condition: it suffices to check that $Z_\pm$ are orthogonal to $\Phi_{0}$, the generator of the kernel of $H$. Indeed, we have from (\ref{Zpm}), (\ref{qH}), the self-adjointedness of $H$ and Proposition \ref{PropH},
\[
\langle \Phi_{0}|Z_\pm \rangle =  \pm \frac\ga{\sqrt{\la_0}} \langle \Phi_{0} | \q H Z_\pm \rangle  =\mp\frac\ga{\sqrt{\la_0}} \langle \Phi_{0} | H J Z_\pm \rangle  =0.
\]
However, we need some additional estimates on $Y_\pm$. In what follows, we write down explicitly the equation $HY_\pm =Z_\pm$. It is not difficult to check that $Y_\pm = (Y_{\pm,1}, Y_{\pm,2})^T$ satisfies the equations
\[
T Y_{\pm,1} -\beta (Y_{\pm,2})_x = Z_{\pm,1},    \qquad   \beta (Y_{\pm,1})_x + Y_{\pm,2} = Z_{\pm,2}.
\]
Replacing the second equation in the first one, we get (cf. \eqref{Lp})
\[
L_\bt^+ Y_{\pm,1} = \bt (Z_{\pm,2})_x + Z_{\pm,1}.
\]
Note that $(\bt (Z_{\pm,2})_x + Z_{\pm,1}| \partial_x Q_\bt) =0$. Therefore, $Y_{\pm,1}$ exists and it is exponentially decreasing, with the same rate as $Z_{\pm,1} $ and $Z_{\pm,2}$. A similar conclusion follows for $Y_{\pm,2}$.

\medskip

Since $Y_{\pm}$ is unique modulo the addition of a constant times $\Phi_{0}$,  it is clear that we can choose $Y_\pm$ such that $\langle \Phi_{0} | Y_\pm \rangle=0$. On the other hand, from Lemma \ref{Z},
\begin{align*}
\langle Y_\pm | HY_\pm \rangle & = \langle Y_\pm | Z_\pm \rangle =\pm \frac{\ga}{\sqrt{\la_0}}\langle Y_\pm | \q H Z_\pm \rangle = \mp \frac{\ga}{\sqrt{\la_0}} \langle HY_\pm | J Z_\pm \rangle \\
& = \mp \frac{\ga}{\sqrt{\la_0}} \langle Z_\pm | J Z_\pm \rangle =0.
\end{align*}
\end{proof}

The main result of this section is the following alternative to Proposition \ref{PropH}.

\begin{prop}\label{prop:22}
There exists $\mu_0 >0$ such that the following holds. Let $V \in H^1 \times L^2$ such that $\langle \Phi_{0}| V \rangle =0$. Then
\[
\langle H  V|V \rangle \geq \mu_0 \|V\|^2 - \frac 1{\mu_0} \Big[ \langle Z_+| V\rangle^2 + \langle Z_-| V\rangle^2\Big] .
\]
\end{prop}

\begin{proof}
It is enough to prove that $\langle \Phi_{0}| V\rangle =  \langle Z_+| V\rangle =  \langle Z_-| V\rangle =0$ imply
\[
\langle H  V|V\rangle \geq \mu_0 \|V\|^2,
\] 
for some $\mu_0>0$, independently of $V$. In order to prove this assertion, we first assume $\bt\neq 0$ and decompose orthogonally $V$ and $Y_\pm$ (cf. the previous Lemma) as follows
\be\label{decoa}
V = \tilde V  + \al_- \Phi_-  + \al_0 \Phi_0, \quad Y_\pm = \tilde Y_\pm  + \delta_0 \Phi_0  + \delta_\pm \Phi_- ,
\ee
 with 
\be\label{decob}
\langle \tilde V | \Phi_0 \rangle = \langle \tilde Y_\pm | \Phi_0 \rangle= \langle \tilde V | H\Phi_- \rangle =\langle \tilde Y_\pm | H \Phi_- \rangle=0 .
\ee
Since $\langle \Phi_{0}| \Phi_- \rangle =\langle \Phi_{0}| V \rangle =\langle \Phi_{0}| Y_\pm \rangle=0$ and $\langle \Phi_- | H\Phi_- \rangle<0$, it is clear that $\al_0 =\delta_0=0$ and $\al_-, \delta_\pm$ are well-defined. Moreover, 

\medskip

\begin{claim} For all $\bt \in (-1,1)\backslash\{0\}$,
 $\tilde Y_+$ and $\tilde Y_-$ are linearly independent as $L^2(\R)^2$ vector-valued functions with real coefficients. 
\end{claim}

\medskip

Indeed, to see this, assume that there is $\tilde \la \neq 0$ such that $\tilde Y_+ = \tilde \la \tilde Y_-$. Then, from the previous decomposition and Lemma \ref{Ypm},
\be\label{star}
Z_+ - \tilde \la Z_- = H (Y_+ -\tilde \la Y_-)  = (\delta_+ - \tilde \la \delta_-) H \Phi_-.
\ee
This identity contradicts (\ref{decr1}) and (\ref{Phi0}), which establish that $Z_+$ and $Z_-$ have essentially different rates of decay at infinity, different to that of $H\Phi_-$, for all $\bt \neq 0$, which makes (\ref{star}) impossible.

\medskip

The analysis is now similar to that in \cite[Lemma 5.2]{DM2}. We have from (\ref{decoa}),
\be \label{eq:prop2_2}
\langle HV|V\rangle =  \langle H\tilde V +  \al_- H\Phi_-  |\tilde V + \al_- \Phi_-\rangle = \langle H\tilde V|\tilde V\rangle +  \al_-^2 \langle H \Phi_- |\Phi_- \rangle. 
\ee
On the other hand, since $\langle Z_\pm |V \rangle =0$, we have from Lemma \ref{Ypm},
\[
0 =\langle Y_\pm |H V\rangle =  \langle \tilde Y_\pm + \delta_\pm \Phi_- |H \tilde V  +  \al_- H\Phi_-\rangle  = \langle \tilde Y_\pm |H \tilde V \rangle  +  \al_- \delta_\pm \langle H \Phi_- |\Phi_- \rangle.
\]
Similarly,
\[
0= \langle HY_\pm | Y_\pm \rangle = \langle H\tilde Y_\pm|\tilde Y_\pm \rangle + \delta_\pm^2 \langle H \Phi_- |\Phi_- \rangle. 
\]
We get then
\be \label{eq:prop2_1}
 \langle H V|V \rangle = \langle H\tilde V|\tilde V\rangle - \frac{  \langle \tilde Y_- |H \tilde V \rangle \langle \tilde Y_+ |H \tilde V \rangle }{\sqrt{ \langle H\tilde Y_+|\tilde Y_+\rangle \langle H\tilde Y_-|\tilde Y_-\rangle}}.
\ee
Consider 
\[ a := \sup_{W \in \Span( \tilde Y_+, \tilde Y_- ) \setminus \{ \vec 0 \}} \left| \frac{  \langle \tilde Y_+ |H \tilde W \rangle}{\sqrt{\langle H\tilde Y_+|\tilde Y_+\rangle \langle H W| W \rangle} } \cdot \frac{\langle \tilde Y_- |H \tilde W \rangle }{\sqrt{ \langle H\tilde Y_-|\tilde Y_-\rangle \langle H W|W \rangle }} \right|. \]
Recall $\langle H \cdot | \cdot \rangle$ is positive definite on $ \Span( \Phi_0, \Phi_- )^\perp$. Hence apply Cauchy Schwarz's inequality to both terms of the product: it transpires that $a \le 1$. Furthermore, if $a=1$ (as $\Span( \tilde Y_+, \tilde Y_- )$ is finite dimensional), there exists $W$ of norm $1$ such that both terms are in the equality case in the Cauchy-Schwarz inequality, i.e $W$ and $\tilde Y_+$ are linearly dependent, and $W$ and $\tilde Y_-$ are also linearly dependent. But it would then follow that $\tilde Y_+$ and $\tilde Y_-$ are linearly dependent, a contradiction the above claim. This proves $a<1$. 

\smallskip

Now using $H$-orthogonal decomposition on $\Span( \Phi_0, \Phi_- )^\perp$, we deduce that
\[ \forall  \ W \in \Span( \Phi_0, \Phi_- )^\perp, \quad \left| \frac{  \langle \tilde Y_- |H  W \rangle \langle \tilde Y_+ |H W \rangle }{\sqrt{ \langle H\tilde Y_+|\tilde Y_+\rangle \langle H\tilde Y_-|\tilde Y_-\rangle}} \right| \le a \langle H W| W \rangle. \]

By \eqref{eq:prop2_1}, \eqref{decob} and \eqref{coer1}, we get
\[ \langle H V|V \rangle \ge (1-a) \langle H \tilde V | \tilde V \rangle \ge \alpha_0 (1-a) \| \tilde V \|^2 \ge 0 . \]
and so \eqref{eq:prop2_2} implies $\langle H \tilde V | \tilde V \rangle \ge \alpha_-^2 |\langle H \Phi_- | \Phi_- \rangle|$.

We then conclude that, for $C = \frac{4}{(1-a)} \max \left( \frac{1}{\al_0}, \frac{\| \Phi_- \|^2}{|\langle H \Phi_- | \Phi_- \rangle|} \right)$,
\begin{align*}
 C \langle H V|V \rangle &  \ge C  (1-a) \langle H \tilde V | \tilde V \rangle \\
& \ge \frac{C(1-a)}{2} (\langle H \tilde V | \tilde V \rangle  + \alpha_-^2 |\langle H \Phi_- | \Phi_- \rangle| )\\
& \ge 2 \| \tilde V \|^2 + 2 \alpha_-^2 \| \Phi_- \|^2 \ge \| \tilde V + \alpha_- \Phi_- \|^2 = \| V \|^2. 
\end{align*}

\medskip

Finally, if  $\bt =0$, we proceed as follows. First of all, we have from (\ref{decr1}) and (\ref{Phi0}),
\[
Z_\pm =  Q^- \Big(\begin{array}{c} \pm  \sqrt{\la_0} \\ 1  \end{array} \Big),  \quad  \Phi_0 = Q' \begin{pmatrix}
1 \\
0
\end{pmatrix}, \quad 
\]
so that $ \langle \Phi_0 |V \rangle = \langle Z_\pm |V \rangle =0$ imply $ (v_1| Q')=(v_1 | Q^- ) =(v_2 | Q^- )=0 $, where $V =(v_1,v_2)^T$.  Therefore,
\[
\langle H V|V \rangle   = (L^+ v_1| v_1) + (  v_2 |  v_2)  \ge\nu_0 \|V\|^2.
\]
\qedhere
\end{proof}

\subsection{Extension to higher dimensions}

The equivalent of Lemma \ref{lemL} (and therefore assumption $(iv)$ of the Introduction) in dimension $d\geq 2$ has the form 

\begin{lem} \label{lemLd} 
Assume $d \ge 2$ and assumption $(B)$ holds.
$L^+$ has exactly one negative eigenvalue, and its kernel is spanned by $(\partial_{x_i} Q)_{i=1,\dots, d}$. Its continuous spectrum is $[1,+\infty)$.
\end{lem}

\begin{proof} 
See Maris \cite{M} and McLeod \cite{ML}.
\end{proof}

As mentioned in the Introduction, this result is open for general nonlinearity $f$. In that case, we need to assume that it holds, i.e. assumption $(v)$.

\medskip

The null directions for $H$ are now the $d$-dimensional vector space spanned by the functions $\Phi_{0,i} = \begin{pmatrix}
\partial_{i} Q \\ 
\beta . \nabla \partial_i Q
\end{pmatrix}$.
In the proof of Lemma \ref{Z}, one should rather perform the transformation $\tilde Z_2 = Z_2 e^{-\gamma \lambda \beta \cdot x}$. The rest of the arguments is dimension insensitive.

\bigskip

\section{Construction of approximate $N$-solitons}

\medskip

In this section we prove Theorem \ref{MT1}. Again, we will give a detailed proof in the one dimensional case $d=1$, and point out how to extend the proof in higher dimension, which is done in a similar fashion as in \cite{Martel1}. 

\subsection{The topological argument}

We continue with the same notation as in the previous section. In particular, we fix $\beta \in (-1,1)$ and consider now the time-dependent, boosted soliton given by
\[ 
Q_\beta(t,x) = Q(\gamma(x-\beta t)), \quad \gamma =(1-\beta^2)^{-1/2}. 
\]
Additionally, we suppose given $N$ different velocities $\bt_1,\ldots, \bt_N\in (-1,1)$, already arranged in such a way that
\be\label{velos}
 -1<\beta_1 <\beta_2 <\ldots < \beta_N<1, 
\ee
and $N$ translation parameters $x_1,\ldots,x_N \in \R$, such that  $Q_{\bt_j}(t,x-x_j)$ is the associated soliton solution of velocity $\bt_j$ and shift $x_j$, $j=1,\ldots, N$. 

\medskip

Finally, we introduce some notation. Given $\q B$ a real Banach space, $x \in B$ and $r \ge 0$, we denote 
\[
B_{\q B}(x,r) = \{ y \in \q B \ | \ \|x-y \|_{\q B} \le r \}
\]
 the \emph{closed} ball in $\q B$ centered  at $x$ of radius $r$ and $\|\cdot \|_{\q B}$ is the associated Banach norm on $\q B$.

\medskip

\begin{lem}[Modulation] \label{mod}
There exist $L_0> 0$ and $\e_0 >0$ such that the following holds for some $C>0$. 
For any $L \ge L_0$ and $0< \e < \e_0$, $t_0 \in \R$,  if $U \in H^1(\R) \times L^2(\R)$ is sufficiently near a sum of solitons whose centers are sufficiently far apart,
\[ 
\left\| U - \sum_{j=1}^N \left( \begin{array}{c} Q_{\beta_j} \\  \partial_{t}Q_{\beta_j} \end{array} \right) (t_0,\cdot -  y_j) \right\| \le \e, \quad \min \big\{ |  y_j -  y_i  | \ | \ i \ne j \big\} \ge L, 
\]
then there exist shifts $\tilde y_j=\tilde y_j(\bt_j,t_0)$ such that if we define
\begin{align}\label{defRj}
 & \tilde R_j(x)  :=\begin{pmatrix} Q_{\beta_j} \\  \partial_{t}Q_{\beta_j} \end{pmatrix}  (0, x- \tilde y_j), \\
 & \tilde R(x) := \sum_{j=1}^N \tilde R_j(x), \label{defR}\\
 &  V(x) := U(x) -\tilde R(x), \label{defV}
\end{align}
then
\be\label{Orthooo}
\| V \| \le C \e, \quad \text{and} \quad \langle V | (\tilde R_j)_x \rangle =0. 
\ee
Also, the map $U \mapsto (V,(\tilde y_j)_j)$ is a $\q C^1$-diffeomorphism around $\sum_{j=1}^N \begin{pmatrix} Q_{\beta_j} \\  \partial_{t}Q_{\beta_j} \end{pmatrix} (t_0, x- y_j)$. In such case, we say that $U$ can be modulated into $(V,(\tilde y_j)_j)$. 
\end{lem}

\begin{proof} This is the classical modulation result, stated as in \cite[Lemma
    2]{CMM}. See \cite{We1,We2} for more details.
\end{proof}

In what follows, we introduce additional notation. We assume that $U$ can be modulated into $(V, (\tilde y_j)_j)$. 

For any $j=1,\ldots, N$ (cf. Proposition \ref{PropH} and Lemma \ref{Z} for the definitions), let
\be\label{EigenJ}
\begin{cases}
Z_{\pm,j}(s) := Z_{\pm}(\ga_j (s- \tilde y_j)),  \quad Z_{0,\beta_j}(s) := Z_{0}(\ga_j (s -\tilde y_j)), \\
\Phi_{0,j}(s) := \Phi_0(\ga_j (s - \tilde y_j)), \quad \Phi_{-,j}(s) := \Phi_-(\ga_j (s- \tilde y_j)), 
\end{cases}
\ee
where $\ga_j:=(1-\bt_j^2)^{-1/2}$, and
\be\label{def:a}
\begin{cases}
a_{\pm,j} := \langle V | Z_{\mp,j} \rangle, \\
a_{0,j} := \langle V | Z_{0,j} \rangle,
\end{cases}
\ee
along with the vectors 
\be\label{avectors}
\boldsymbol a_+ = (a_{+,j})_j, \quad \boldsymbol a_- = (a_{-,j})_j, \quad \boldsymbol a_{0} = (a_{0,j})_j,  \quad \hbox{and} \quad   \boldsymbol{\tilde y} = (\tilde y_{j})_j .
\ee
Finally, we fix a constant $\ga_0$ given by
\bea\label{ga0}
\ga_0 &:= & \min\Big\{ \frac 14\sqrt{\la_0}   \min \{ \frac 1\ga_1, \frac1\ga_2,\ldots,\frac1\ga_{N} \}  , \nonu\\
& &  \qquad \frac 14 \min \{ \ga_1, \ga_2,\ldots,\ga_{N} \} \min \{ \bt_1, \bt_2-\bt_1,\ldots,\bt_N-\bt_{N-1} \} \Big\}>0.
\eea
Assume now $\ve \in (0,\ve_0)$ and $L\geq L_0$, where $\ve_0$ and $L_0$ are obtained in by Lemma \ref{mod}. Given $t\in \R$, let us consider the \emph{centers}
\[
 y_j = y_j(t):= \beta_j t + x_{j}, \quad j=1,\ldots,N,
\] 
where the velocities $\bt_j$ and the shifts $x_j$ are given by \eqref{velos}. It is clear that there exists $T_0 \in \R$ such that, for all $t\geq T_0$, the $ y_j$ satisfy 
\[
\min \big\{ | y_j - y_i| \ |  \ i\neq j \big\} \geq L.
\]
From now on, we fix $t\geq T_0$. Consider the corresponding sum of solitons $ R(t,x)$ associated to these parameters, namely
\be\label{Rdt}
R(t,x) := \sum_{j=1}^N R_j(t,x) = \sum_{j=1}^N \left( \begin{array}{c} Q_{\beta_j} \\  \partial_{t}Q_{\beta_j} \end{array} \right) (0, x-  y_j).
\ee
Then, according to Lemma \ref{mod}, if $U \in H^1(\R) \times L^2(\R)$ satisfies $\| U - R(t) \| \le \e$, then $U$ can be modulated. Moreover, up to increasing $T_0$, we can assume that $e^{-\ga_0 T_0} < \e_0$. Thus we can define our shrinking set.

\begin{defi}[Shrinking set $\q V(t)$]
For $t \ge T_0$, we define the set 
\[ \q V(t) \subset B_{H^1 \times L^2}(R(t), \e_0) \]
in the following way: $U \in \q V(t)$ if and only if $U$ can be modulated into $(V,\boldsymbol{\tilde y})$ where (cf. \eqref{defRj} and \eqref{avectors})
\[
V(t) = U - \sum_{j=1}^N \tilde R_j(t), 
\]
with
\begin{gather}\label{hyp}
\| V(t) \| \le e^{-\ga_0 t}, \qquad | \tilde y_{j}(t) - \beta_j t - x_j| \le e^{-\ga_0 t}, \\
\smallskip
\| \boldsymbol a_+ \|_{\ell^2} \le e^{-3\gamma_0 t/2}, \quad \| \boldsymbol a_- \|_{\ell^2} \le e^{-3\gamma_0 t/2}, \quad \| \boldsymbol a_0 \|_{\ell^2} \le e^{-3\gamma_0 t/2}. \label{hyp2}
\end{gather}
\end{defi}

\begin{defi}
We denote by $\varphi = (u,\partial_t u)^T$ the flow of the (NLKG) equation, that is, given $S_0 \in \R$ and $U_0 \in H^1(\R) \times L^2(\R)$, 
\be\label{phi}
t \mapsto \varphi(S_0,t,U_0)
\ee
 is the solution to (NLKG) with initial data $U_0$ at time $S_0$
 (with values in $H^1 \times L^2$). 
\end{defi}

In most of what we do, we will have $t \le S_0$ so that $U_0$ can be thought of as a final data, and we work backwards in time. The key result of this section is the following construction of an approximate $N$-soliton.

\begin{prop}[Approximate $N$-soliton] \label{prop:app_Nsol}
There exist $T_0 >0$ such that the following holds. For any $S_0 \ge T_0$, there exist a final data $U_0$ such that
\[ \forall t \in [T_0,S_0], \quad \varphi(S_0,t,U_0) \in \q V(t). \]
\end{prop}

At this point, the solution $\phi(S_0,t,U_0)$ depends on $S_0$. To prove Theorem \ref{MT1}, we will need to derive such a solution independent of $S_0$, which we will do via a compactness argument in the next (and last) section \ref{sec:4}.

\medskip

Our goal is now to prove Proposition \ref{prop:app_Nsol}.

\medskip

Fix $S_0 \ge T_0$. Consider an initial data $U_0$ at time $S_0$ such that $U_0 \in \q V(S_0)$. Due to the blow-up criterion for (NLKG), and the fact that $R(t)$ defined in \eqref{Rdt} is bounded in $H^1(\R) \times L^2(\R)$, we have that $\varphi(S_0,t,U_0)$ is defined at least as long as it belongs to $B_{H^1\times L^2}(R(t),1)$. In particular, $\varphi(S_0,t,U_0)$ does not blow-up as long as it belongs to $\q V(t)$, and we can define the (backward) exit time
\[ 
T^*(U_0) := \inf \left\{ T \in [T_0,S_0] \ \middle| \ \forall t \in [T,S_0], \ \varphi(S_0,t,U_0) \in \q V(t) \right\}. 
\]
Notice that we could have $T^*(U_0) = S_0$. Our goal is to find $U_0 \in \q V(T_0)$ such that $T^*(U_0) = T_0$.

\medskip

In order to show such an assertion, we will only consider some very specific initial data, namely $U_0 \in \q V(S_0)$ such that  (see \eqref{avectors})
\begin{itemize}
\item $U_0 \in R(S_0) + \Span((Z_{\pm,j})_{j =1, \dots, N})$,
\smallskip

\item $\boldsymbol a_-(S_0) = \boldsymbol a_0(S_0) =0$, and 

\smallskip

\item $\boldsymbol a_+(S_0) \in B_{\m R^N}(0,e^{-3\ga_0 S_0/2})$.
\end{itemize}
These conditions can be satisfied due to the almost orthogonality of $Z_{\pm,j}$, $Z_{0,j}$, and this is the content of the following

\begin{lem}[Modulated final data]\label{ModData}
Let $S_0 \ge T_0$ be large enough. There a exist a  $\q C^1$ map $\Theta: B_{\R^N}(0, 1) \to \q V(S_0)$ as follows. Given $\boldsymbol{\gh a}_+ = (\gh a_{+,j})_j \in B_{\R^N}(0, 1)$, $U_0 =: \Theta (\boldsymbol{\gh a}_+) \in \q V(S_0)$ such that $U_0$ can be modulated into $(V_0,\boldsymbol{\tilde y})$ and the associated parameters \eqref{avectors} satisfy 
\be\label{avectors2}
\boldsymbol a_+(S_0) = e^{-3\gamma_0 T_0/2} \boldsymbol{\gh a}_+, \quad \boldsymbol a_- (S_0)= \boldsymbol a_0(S_0) =0.
 \ee
Moreover,
\be\label{V0}
\|V_0\| \leq Ce^{-3\gamma_0 T_0/2}.
\ee
\end{lem}

\begin{proof}
The main idea is to consider the map $B_{\R^{2N}}(0, 1) \to B_{\R^{2N}}(0, 1)$, $\boldsymbol{\gh b}_{\pm} \mapsto \boldsymbol{\gh a}_{\pm}$, where $\boldsymbol{\gh a}_{\pm}$ corresponds to the data $U_0 = R(S_0) + \sum_{\pm,j} \gh b_{\pm,j} Z_{\pm,j}$, and to invoke the implicit mapping theorem. We refer to \cite[Lemma 3]{CMM} and its proof in \cite[Appendix A]{CMM} for full details.
\end{proof}

If $T^*:=T^*(U_0)>T_0$, by maximality, we also have that for the function $\varphi(S_0,
T^*,U_0)$, at least one of the inequalities in the definition of $\q
V(T^*)$ is actually an equality. It turns out that the equality is achieved
by $\boldsymbol a_+(T_0)$ only, and that the rescaled quantity $e^{3\ga_0
  T^*/2} \boldsymbol a_+(t)$ is \emph{transverse} to the sphere at $t=
T^*$. This is at the heart of the proof and is the content of the following

\begin{prop}\label{prop:flow}
Let $\boldsymbol{\gh a}_+ \in B_{\R^N}(0,1)$, and assume that its maximal exit time is (strictly) greater that $T_0$:
\[ T^* = T^*(\Theta(\boldsymbol{\gh a}_+)) > T_0. \]
Denote, for all $t \in [T^*,S_0]$, the associated modulation $(V(t), \tilde{\boldsymbol y}(t))$  of $\varphi(t,T_0,\Theta(\boldsymbol{\gh a}_+))$, defined in (\ref{phi}). Then, for all $t \in [T^*,S_0]$,
\begin{gather} 
\| V(t) \| \le \frac{1}{2} e^{-\ga_0 t}, \quad | \tilde y_j(t) - \beta_j t - x_{j} | \le \frac12 e^{-\ga_0 t}, \label{fin1}\\
 \| \boldsymbol a_- (t) \|_{\ell^2} \le \frac{1}{2} e^{-3\ga_0 t/2}, \quad \| \boldsymbol a_0 (t) \|_{\ell^2} \le \frac{1}{2} e^{-3\ga_0 t/2},\label{fin2}
\end{gather}
and
\be\label{last1} 
\| \boldsymbol a_+ (T^*)\|_{\ell^2} = e^{-3\ga_0 T^*/2}.
\ee
Furthermore, $\boldsymbol a_+(T^*)$ is transverse to the sphere, i.e., 
\[
 \left. \frac{d}{dt} ( e^{3\gamma_0 t} \| \boldsymbol a_+(t) \|_{\ell^2}^2)
 \right|_{t = T^*} < 0.  
\]
\end{prop}

For the sake of continuity, we postpone the proof of Proposition
\ref{prop:flow} until the next paragraph, and conclude the proof of Proposition
\ref{prop:app_Nsol} here, assuming Proposition \ref{prop:flow}.

\medskip

Let us state a few direct consequences of Proposition \ref{prop:flow}, (their proofs will also be done in the next paragraph \ref{4}).

\begin{cor} \label{cor:flow} We have the following properties.
\begin{enumerate}
\item The set of final data which give rise to solutions which exit strictly after $T_0$
\[ \Omega := \{ \boldsymbol{\gh a}_+ \in  B_{\R^N}(0, 1) \mid
T^*(\Theta (\boldsymbol{\gh a}_+)) >T_0 \} \]
is open $($in $B_{\R^N}(0, 1)$$)$.

\smallskip

\item The map $\Omega \to \m R$, $\boldsymbol{\gh a}_+ \mapsto T^*(\Theta
  (\boldsymbol{\gh a}_+)) \in \m R$ is continuous (we emphasize that the data belong to $\Omega$).
  
  \smallskip
  
\item The exit is instantaneous on the sphere:
\be \label{TstarT0}
\text{if} \quad \| \boldsymbol{\gh a}_+ \|_{\ell^2} =1, \quad \text{then}
\quad T^*(\boldsymbol{\gh a}_+) = S_0.
\ee
\end{enumerate}
\end{cor}

We are now in a position to complete the proof of Theorem \ref{MT1}.

\begin{proof}[End of the proof of Proposition \ref{prop:app_Nsol}]
We argue by contradiction. Assume that all possible $\boldsymbol{\gh
  a}_+\in B_{\R^N}(0, 1)$ give rise to initial data $U_0=\Theta
(\boldsymbol{\gh a}_+) \in \q V(S_0)$ and corresponding solutions
$\varphi(S_0,t,U_0)$ that exit $\q V(t)$ strictly after $T_0$, i.e.
\be\label{exit}
\text {assume that } \Omega =  B_{\R^N}(0, 1).
\ee
Given $U_0 \in \q V(T_0)$, we denote $\Phi(U_0)$ the rescaled projection of the \emph{exit spot}
\[
 \Phi(U_0) = e^{3\ga_0 T^*(U_0)/2} \boldsymbol a_+(T^*(U_0)), 
\]
so that $\Phi(U_0) \in B_{\R^N}(0,1)$. Let us finally consider the rescaled projection of the exit spot $\Psi$, defined as follows:
\[ 
\Psi: B_{\R^N}(0, 1) \to  \m B_{\R^N}(0,1), \quad \boldsymbol{\gh a}_+ \mapsto  \Psi(\boldsymbol{\gh a}_+) =\Phi \circ \Theta (\boldsymbol {\gh a}_+). 
\]
Corollary \ref{cor:flow} then translates into the following properties for $\Psi$:
\begin{itemize}
\item $\Psi: B_{\R^N}(0, 1) \to  \m S^{N-1}$ is continuous (like $T^*$, $\Phi$ and $\Theta$);

\smallskip

\item If $\| \boldsymbol{\gh a}_+ \|_{\ell^2} =1$, $\Psi(\boldsymbol{\gh a}_+) = \boldsymbol{\gh a}_+$ (cf. (\ref{TstarT0}) and (\ref{avectors2})); i.e $\Psi|_{\m S^{N-1}} = \Id$.
\end{itemize}
These two affirmations contradict the Brouwer's Theorem. Hence our
assumption \eqref{exit} is wrong, and there exists $\gh a^+$ such that the
solution $U(t) = \varphi(S_0,t,\Theta (\boldsymbol {\gh a}_+))$ satisfies $T^+(\Theta (\boldsymbol {\gh a}_+))= T_0$. In particular $U(t) \in \q V(t)$ for all $t \in [T_0,S_0]$, and $U_0:=U(S_0)=\Theta (\boldsymbol {\gh a}_+) $ satisfies the
conditions of Proposition \ref{prop:app_Nsol}.
\end{proof}

\subsection{Bootstrap estimates}\label{4}

This paragraph is devoted to the last remaining results needed to complete Proposition \ref{prop:app_Nsol}: Proposition \ref{prop:flow} and Corollary \ref{cor:flow}

\begin{proof}[Proof of Proposition \ref{prop:flow}]

\noindent
{\bf Step 1.}  First, we introduce some notation. Consider the flow $\varphi(t) = \varphi(S_0,t,\Theta(\boldsymbol{\gh a}_+))$ given by Proposition \ref{prop:flow}, and valid for all $t\in [T^*,S_0]$. From Lemma \ref{mod}, we have
\be\label{deco}
\varphi(t) =\tilde R(t) + V(t), 
\ee
where
\begin{gather}
\label{yjxj}
\tilde R(t,x)  = \sum_{j=1}^N \tilde R_j(t,x) , \quad  \tilde R_j(t,x) = (Q_{\beta_j}, \partial_{t}Q_{\beta_j})^T (x- \tilde y_j(t)),  \\ \tilde y_j(t) = \bt_j t+ \tilde x_j(t), \label{xy}
\end{gather}
and 
\[
V(t) = (v_1(t),v_2(t))^T.
\]
Additionally, from  the equation satisfied by $\varphi$, we have 
\[
\varphi_t = \left( \begin{array}{cc} 0 &  \Id \\   \partial_x^2 - \Id & 0 \end{array}\right) \varphi  +\left( \begin{array}{c} 0 \\ f(u) \end{array}\right), 
\]
where $\varphi = (u,u_t)^T$. Replacing the decomposition (\ref{deco}), we have
\be\label{eqV}
V_t = \left( \begin{array}{cc} 0 &  \Id \\  \partial_x^2- \Id  + f'(Q_{\bt_j}) & 0 \end{array}\right) V   +   Rem(t) = \mathcal L_j V   +   Rem(t), 
\ee
with $\mathcal L_j :=\mathcal L(\beta_j)$ defined in (\ref{qL}),
\begin{align*}
Rem(t) & := \begin{pmatrix} 0 &  \Id \\  \partial_x^2- \Id & 0 \end{pmatrix} \tilde R -\tilde R_t + \begin{pmatrix} 0 \\ f(u) -f'(Q_{\bt_j})v_1\end{pmatrix} \\
& = \sum_{k=1}^N  \tilde x_k'(t) \partial_x \begin{pmatrix} Q_{\bt_k} \\ -\bt_k \partial_x Q_{\bt_k} \end{pmatrix} \\
& \quad + \begin{pmatrix} 0 \\   f(\sum_{k=1}^N Q_{\bt_k} + v_1) - \sum_{k=1}^N f(Q_{\bt_k}) - f'(Q_{\bt_j})v_1 \end{pmatrix}. 
\end{align*}
First of all, note that from \eqref{Phi0} we have $ \begin{pmatrix} \partial_xQ_{\bt_k} \\ -\bt_k \partial_{xx} Q_{\bt_k} \end{pmatrix} =  \Phi_{0,k}.$ If we take the scalar product of \eqref{eqV} with $(\tilde R_j)_x$, then the orthogonality \eqref{Orthooo} (coming from modulation) leads to the estimate
\be\label{xjxj}
|\tilde x_j'(t)| \le C(\|V(t)\| + e^{-3\ga_0t}),
\ee
valid for all $j=1,\ldots, N$. Indeed, we have
\[
\langle (\tilde R_{j})_x |V_t \rangle =\langle (\tilde R_{j})_x | \mathcal L_j V \rangle  +\langle (\tilde R_{j})_x |Re(t) \rangle .
\]
Note that from \eqref{xy}
\[
\langle (\tilde R_{j})_x |V_t \rangle = -\langle (\tilde R_{j})_{xt} |V \rangle  = (\bt_j + \tilde x_j'(t)) \langle \partial_{xx} \tilde R_{j} |V \rangle.
\]
Consequently
\[
|\langle (\tilde R_{j})_x |V_t \rangle | \leq C (1+|\tilde x_j'(t)|) \|V(t)\|.
\]
On the other hand,
\[
\langle (\tilde R_{j})_x | \mathcal L_j V \rangle = \left\langle \begin{pmatrix} \partial_x Q_{\beta_j} \\  \partial_{xt}Q_{\beta_j} \end{pmatrix}   \Big|  \begin{pmatrix} v_2 \\ - T_{\beta_j} v_1 \end{pmatrix}   \right\rangle = \left\langle \begin{pmatrix} \partial_x Q_{\beta_j} \\ - T_{\beta_j} \partial_{xt}Q_{\beta_j} \end{pmatrix}   \Big|  \begin{pmatrix} v_2 \\  v_1 \end{pmatrix}   \right\rangle,
\]
so that 
\[
|\langle (\tilde R_{j})_x | \mathcal L_j V \rangle | \leq C \|V(t)\|.
\]
Finally, we deal with the term $\langle (\tilde R_{j})_x |Rem(t) \rangle $. From the definition of $Rem(t)$ we have
\bee
\langle (\tilde R_{j})_x |Rem(t) \rangle & = & \sum_{k=1}^N  \tilde x_k'(t) \langle  (\tilde R_{j})_x |  \Phi_{0,k} \rangle \\
& &  \quad + \Big( \partial_{xt} Q_{\bt_j} \Big| f \Big(\sum_{k=1}^N Q_{\bt_k} + v_1 \Big) - \sum_{k=1}^N f(Q_{\bt_k}) - f'(Q_{\bt_j})v_1 \Big).
\eee
Since $\tilde R_{j})_x = \Phi_{0,j}$, we get
\[
\langle  (\tilde R_{j})_x |  \Phi_{0,j} \rangle = \| \Phi_{0,j}\|^2,
\]
and if $k\neq j$, 
\[
|\langle  (\tilde R_{j})_x |  \Phi_{0,k} \rangle | =| \langle \Phi_{0,j}|  \Phi_{0,k} \rangle | \leq Ce^{-3\ga_0 t}.
\]
Now if $x \in [ m_{j}t, m_{j+1}t ]$,  then for all $p\neq j$ (see \eqref{ga0}),
\[
| Q_{\beta_p}(t,x)| \leq Ce^{-3\ga_0 t}.
\]
Therefore, inside this region (note that if $d\geq 2$ then $f$ is a pure power nonlinearity)
\[
\Big\| f \Big(\sum_{k=1}^N Q_{\bt_k} + v_1 \Big) - \sum_{k=1}^N f(Q_{\bt_k}) - f'(Q_{\bt_j})v_1 \Big\|_{L^2} \leq Ce^{-3\ga_0 t}+ C\|V(t)\|^2.
\]
On the other hand, if $x \notin [ m_{j}t, m_{j+1}t ]$
\[
 | \partial_{xt} Q_{\beta_j}| \leq Ce^{-3\ga_0 t}.
\]
In conclusion, we have
\be\label{cle}
\Big( \partial_{xt} Q_{\bt_j} \Big| f \Big(\sum_{k=1}^N Q_{\bt_k} + v_1 \Big) - \sum_{k=1}^N f(Q_{\bt_k}) - f'(Q_{\bt_j})v_1 \Big) \leq C e^{-3\ga_0 t} + C\|V(t)\|^2.
\ee
Collecting the preceding estimates we get \eqref{xjxj}.
%

\medskip

\noindent
{\bf Step 2. Control of degenerate directions.}  The next step of the proof is to consider the dynamics of the associated scalar products $a_{\pm,j}(t)$ and $a_{0,j}(t)$ introduced in (\ref{def:a}).

\begin{lem}\label{daj}
Let $a_{\pm,j} (t) $ and $a_{0,j}(t)$ be as defined in (\ref{def:a}). There is a constant $C>0$, independent of $S_0$ and $T^*\ge T_0$, such that for all $t\in [T^*, S_0]$,
\be\label{aprime}
\left| a_{\pm,j}'(t)  \pm  \frac{\sqrt{\la_0}}{\ga_j} a_{\pm,j}(t) \right| \le C\|V(t)\|^2 + Ce^{-3\ga_0 t},
\ee
and
\be\label{a0prime}
\abs{ a_{0,j}'(t)} \le C\|V(t)\|^2 + Ce^{-3\ga_0 t}.
\ee
\end{lem}

\begin{proof}
We  prove the case of $a_{-,j}(t)$. The other cases are similar. We compute the time derivative of $a_{-,j}$ using \eqref{yjxj} and \eqref{eqV}, and we choose $\ga_0>0$ as small as needed, but fixed.
\begin{align*}
a_{-,j}'(t) & = - \tilde y_j'(t)  \langle  (Z_{+,j})_x |V(t)\rangle +  \langle Z_{+,j}| V_t (t) \rangle  \\
& = - \tilde x_j'  \langle (Z_{+,j})_x |V(t) \rangle + \langle(\q L_j^* - \bt_j \partial_{x}) Z_{+,j}| V(t) \rangle  + \sum_{k=1}^N x_k'  \langle \Phi_{0,k}|Z_{+,j}\rangle \\
& \quad + O(\|V(t)\|^2 + e^{-3\ga_0 t}).
\end{align*}
From Lemma \ref{Ypm} we have $ \langle \Phi_{0,j}|Z_{+,j} \rangle =0$. Therefore, since $\q L_j^* - \bt_j \partial_{x} = \q H_j$, where $\q H_j := \q H(Q_{\bt_j})$ (cf. (\ref{qH})), we have from Lemma \ref{Z} and \eqref{xjxj},
\begin{align*}
a_{-,j}'(t) &= \frac{\sqrt{\la_0}}{\ga_j} a_{-,j}(t) +  O( |x_j'| \|V(t)\|  + \|V(t)\|^2 +   e^{-3\ga_0t} )\\
& = \frac{\sqrt{\la_0}}{\ga_j} a_{-,j}(t) +  O( \|V(t)\|^2 +   e^{-3\ga_0t} ). \qedhere
\end{align*}
\end{proof}

\noindent
{\bf Step 3. Lyapunov functional.}  Let $L_0>0$ be a large constant to be chosen later. Let $(\phi_j)_{j=1,\dots, N}$ be a partition of the unity of $\R$ placed at the midpoint between two solitons. More precisely, let 
\be\label{phi}
\phi\in C^\infty(\R), \quad \phi'>0, \quad \lim_{-\infty}\phi=0, \quad \lim_{+\infty}\phi=1. 
\ee
We have\footnote{Do not confuse the constant $L$ in \eqref{LLL} with the operator $L$ in \eqref{LL}.}, for all $L>L_0$,
\be \label{LLL}
\sum_{j=1}^N \phi_j (t,x)\equiv 1, \qquad \phi_j (t,x) = \phi \Big( \frac{x -m_j t}{L}\Big) - \phi \Big( \frac{x -m_{j+1} t}{L}\Big),
\ee
where $m_j := \frac 12 (\bt_j+ \bt_{j-1})$, with $j=2,\ldots, N-1$, and $m_1 :=-\infty$, $m_N=+\infty$. We introduce the $j$-th portion of momentum
\be\label{Pj}
P_j[\varphi](t) := \frac 12 \int \phi_j  u_t  u_x \ dx, \quad \varphi = (u,u_t)^T,
\ee
and the modified \emph{Lyapunov functional} 
\be\label{WF}
\mathcal F[\varphi](t) := E[\varphi](t) +  2\sum_{j=1}^N\beta_j P_j [\varphi](t), 
\ee
with $E[\varphi]$ being the energy defined in (\ref{E}). Our first result is a suitable decomposition of $\mathcal F[u]$ around the multi-soliton solution.
\begin{lem}
Let $ V (t) =(v_1(t),v_2(t))^T$ be the error function defined in Proposition \ref{prop:flow}. There is a positive constant $C>0$ such that
\be\label{Expa}
\Big| \mathcal F[\varphi](t) -  \sum_{j=1}^N \langle H_jV|V \rangle  \Big| \le  C \|V (t)\|^3 + \frac CL e^{-2\ga_0 t},
\ee
where
\be\label{HVj}
\langle H_jV|V \rangle  :=  \int  \phi_j (v_2^2 + (v_1)_x^2 + v_1^2 - f'(Q_{\beta_j}) v^2 +2 \bt_j  v_2 (v_1)_x ).
\ee
\end{lem}

\begin{proof}
From the decomposition 
\be\label{desco}
\varphi(t) = (u, u_t)(t) =(\tilde R_1, \tilde R_2)^T + (v_1,v_2)^T(t),
\ee
we have
\begin{align*}
\mathcal F[\varphi](t) &= \frac 12 \int (u_t^2 +u_x^2 +u^2 - F(u)) +  \sum_{j=1}^N\beta_j \int  \phi_j  u_t  u_x  \\
& =  \frac 12 \int (\tilde R_2^2 + (\tilde R_1)_x^2  +\tilde R_1^2  -2F(\tilde R_1) ) + \sum_{j=1}^N \bt_j   \int \tilde R_2  (\tilde R_1)_x \ \phi_j \\
& \quad + \int \big[ \tilde R_2 v_2 + (\tilde R_1)_x  (v_1)_x   + \tilde  R_1 v_1   -  f(\tilde R_1) v_1 + \sum_{j=1}^N \bt_j (\tilde R_2 (v_1)_x + v_2  (\tilde R_1)_x) \phi_j\big] \\
& \quad +\frac 12 \int (v_2^2 + (v_1)_x^2 + v_1^2 -f'(\tilde R_1) v_1^2) +\sum_{j=1}^N \bt_j \int v_2  (v_1)_x \ \phi_j \\
& \quad - \int ( F(\tilde R_1 +v_1) - F(\tilde R_1) - f(\tilde R_1) v_1 - \frac 12 f'(\tilde R_1) v_1^2 )\\
&=:   I_1 + I_2 + I_3 + I_4.
\end{align*}
Let us consider the term $I_1$. Since  $\tilde R_2= -\sum_{j=1}^N \bt_j (Q_{\bt_j})_x$ and $(\tilde R_1)_x =\sum_{j=1}^N  (Q_{\bt_j})_x$, one has
\bee
I_1 & =&  \frac 12 \sum_{j=1}^N  \int  \Big[  \bt_j^2(Q_{\bt_j})_x^2 + (Q_{\bt_j})_x^2+ Q_{\bt_j}^2 - 2F(Q_{\bt_j})  - 2\bt_j^2  (Q_{\bt_j})_x^2  \Big]  + O(e^{-3\ga_0 t})\\
& =& \frac 12 \int  \Big[  Q_x^2 + Q^2 - 2F(Q) \Big]\sum_{j=1}^N\frac 1{\ga_j}  + O(e^{-3\ga_0 t}).
\eee
Now we consider $I_2$. Integrating by parts, we have
\begin{align*}
I_2 & = \int  v_2 \Big[ \tilde R_2 + (\tilde R_1)_x  \sum_{j=1}^N \bt_j   \phi_j \Big]  \\
& \quad - \int v_1 \Big[ (\tilde R_1)_{xx}    - \tilde R_1   +  f(\tilde R_1)  +  (\tilde R_2)_{x}  \sum_{j=1}^N \bt_j  \phi_j\Big] - \sum_{j=1}^N \bt_j  \int  v_1 \tilde R_2 (\phi_j)_x. 
\end{align*}
Note that
\[
\tilde R_2 + (\tilde R_1)_x  \sum_{j=1}^N \bt_j   \phi_j =   \sum_{k=1}^N  [-\bt_k   (Q_{\bt_k})_x +(Q_{\bt_k})_x  \sum_{j=1}^N\bt_j  \phi_j]  =    \sum_{k=1}^N (Q_{\bt_k})_x \sum_{j\neq k}^N\bt_j  \phi_j.
\]
Hence
\[
 \int  v_2 \Big[ \tilde R_2+ (\tilde R_1)_x  \sum_{j=1}^N \bt_j   \phi_j \Big] =  O(e^{-3\ga_0t}).
\]
On the other hand,
\[
 \int v_1 \Big[ (\tilde R_1)_{xx}    - \tilde R_1    + f(\tilde R_1)  +  (\tilde R_2)_{x}  \sum_{j=1}^N \bt_j  \phi_j\Big]  =   O(e^{-3\ga_0t}).
 \]
Finally,
\[ \abs{\sum_{j=1}^N \bt_j \int  v_1 \tilde R_2(\phi_j)_x } \le C \|v_1 \|_{L^2(\R)} e^{-2\ga_0 t}. \]
Gathering the above estimates, we get
\[ |I_2| \le    Ce^{-3\ga_0t}. \]
Let us consider the integral $I_3$. Since $\sum_j \phi_j =1$, we have
\begin{align*}
I_3 &  = \frac 12 \sum_{j=1}^N \int \phi_j (v_2^2 + (v_1)_x^2 + v_1^2 - f'(\tilde R_1) v_1^2 +2 \bt_j  v_2 (v_1)_x ) \\
& = \frac 12 \sum_{j=1}^N \int \phi_j (v_2^2 + (v_1)_x^2 + v_1^2 - f'(Q_{\beta_j}) v_1^2 +2 \bt_j  v_2 (v_1)_x ) - \frac 12 \sum_{k\neq j} \int \phi_j  f'(Q_k) v_1^2 \\
&  \quad - \frac 12 \sum_{j=1}^N \int \phi_j \Big(f'( \sum_{k=1}^N Q_{\beta_k}) -  \sum_{k=1}^N f'(Q_{\beta_k}) \Big) v_1^2   .  
\end{align*}
Fix $\ell\in \{1,\ldots, N-1\}$. If $x \in [ m_{\ell}t, m_{\ell+1}t ]$,  then for all $p\neq \ell$,
\[
|Q_{\beta_p}(t,x)| \leq Ce^{-2\ga_0 t}.
\]
Therefore, for all $x \in [ m_{\ell}t, m_{\ell+1}t ]$,
\[
\abs{f'( \sum_{k=1}^N Q_{\beta_k}(t,x)) -  \sum_{k=1}^N f'(Q_{\beta_k}(t,x) )} \leq Ce^{-\ga_0 t}.
\]
Repeating the same argument for each $\ell$, and using \eqref{fin1}, we get
\[
I_3  = \frac 12 \langle H_jV | V\rangle  + O(e^{-3\ga_0 t}).
\]

Finally, we consider $I_4$. It is not difficult to check that 
\[
|I_4| \le C \|V\|^3.
\]
Collecting the above results, we get finally (\ref{Expa}).
\end{proof}

Our next result describes the variation of the momentum $P_j$.
\begin{lem}
There exists $C>0$ independent of time and $L$, such that for all $t\in [T^*,S_0]$,
\be\label{DPj}
\abs{ P_j[\varphi](t) -P_j[\varphi](S_0)} \le \frac CL e^{-2\ga_0 t}.
\ee
\end{lem}

\begin{proof}
A simple computation using (\ref{NLKG}) shows that 
\begin{align}
\partial_t P_j[\varphi](t)  & = -\frac 14 \int u_t^2 (\phi_j)_x -\frac 14 \int u_x^2 (\phi_j)_x  +\frac 14 \int u^2 (\phi_j)_x \nonu\\
&  - \frac12\int F(u)(\phi_j)_x + \frac 12  \int u_t u_x (\phi_j)_t.
\end{align}
Indeed, one has
\begin{align*}
\partial_t P_j[\varphi](t) & = \frac12 \int u_t  u_x   (\phi_j)_t +\frac12 \int u_t  u_{tx} \phi_j +\frac12 \int u_{tt} u_x  \phi_j\\
& =   \frac12 \int u_t  u_x   (\phi_j)_{t} +\frac14 \int (u_t^2)_x   \phi_j +\frac12 \int (u_{xx} -u + f(u)) u_x  \phi_j\\
& =   \frac12  \int u_t  u_x   (\phi_j)_{t}  - \frac14 \int u_t^2 (\phi_j)_x -\frac 14 \int (u_x^2 -u^2 +2F(u) ) (\phi_j)_x,
\end{align*}
as desired. Now, from the decomposition \eqref{desco} we replace above to obtain (compare with \eqref{cle})
\begin{multline*}
|\partial_t P_j[\varphi](t)| \\
\le C \left( \frac{e^{-3\ga_0 t}}{L} + \int v_2^2 (\phi_j)_x +  \int (v_1)_x^2 (\phi_j)_x  + \int v_1^2 (\phi_j)_x + \int F(v_1)(\phi_j)_x \right) .
\end{multline*}
From the smallness condition of $v$, we get finally
\[ 
|\partial_t P_j[\varphi](t)|  \le \frac CL e^{-2\ga_0 t}, 
\]
as desired. The conclusion follows after integration in time.
\end{proof}

The previous Lemma and the energy conservation law imply the following

\begin{cor}
There exists $C>0$ independent of time and $L>0$ such that, for all $t\in [T^*,S_0]$,
\be\label{monot}
|\mathcal F[\varphi](t) - \mathcal F[\varphi](S_0)| \le  \frac CL e^{-2\ga_0 t}.
\ee
\end{cor}

Now we use the coercivity associated to $H_j$. A standard localization argument (see e.g. \cite{MMT}), Proposition \ref{prop:22}  and (\ref{def:a}) give
\[ 
 \sum_{j=1}^N \langle  H_j V | V \rangle \geq \nu_0 \| V(t)\|^2 - \frac 1{\nu_0} ( \| \boldsymbol a_+ \|_{\ell^2}^2 +\| \boldsymbol a_- \|_{\ell^2}^2 ), 
\]
for an independent constant $\nu_0>0$. From this coercivity estimate, using (\ref{monot}) and \eqref{Expa}, the initial bound  \eqref{V0}, and bounding the terms in $\boldsymbol a_\pm$ by (\ref{hyp2}), we get that for some $C>0$
\[
\forall t \in [T^*,S_0], \quad \| V(t)\| \le \frac{C}{\sqrt L} e^{-\ga_0 t} + Ce^{-3/2\ga_0 t}.
\]
Therefore, for $L \ge 4C^2$, we improve the first condition in (\ref{hyp}), to get (\ref{fin1}). We can now integrate of the modulation equation \eqref{xjxj} for $\tilde x_j'(t)$ we get the second estimates in (\ref{fin1}) (by increasing $L$ is necessary).

\smallskip

Now, using (\ref{aprime})-(\ref{a0prime}) and integrating in time, we improve in a similar way the conditions in (\ref{hyp2}), to obtain (\ref{fin2}).  In conclusion, \eqref{last1} must be satisfied.

\bigskip

\noindent {\bf Step 4. Transversality.} For notation, let $\q N(\boldsymbol{\gh a}_+, t) := e^{3\gamma_0 t} \| \boldsymbol a_+(t) \|_{\ell^2}^2$.
Using the expansion \eqref{aprime}, we compute 
\begin{align*}
\frac{d}{dt} \q N(\boldsymbol{\gh a}_+, t) &= \sum_{j=1}^N e^{3\gamma_0 t} a_{+,j}'(t) a_{+,j}(t) + 3 \gamma_0 \q N(\gh a^+, t) \\
& = -e^{3 \gamma_0 t} \sqrt{\lambda_0} \sum_{j=1}^N \frac{ |a_{+,j}(t)|^2}{\gamma_j} + O( e^{3\gamma_0 t} (\| V(t) \|_{L^2}^2 + e^{-2 \gamma_0 t}) \| \boldsymbol a_+(t) \|_{\ell^2}) \\
& \quad + 3 \gamma_0 \q N(\boldsymbol {\gh a}^+, t) \\
& \le -(2c_0 - 3 \gamma_0) \q N(\boldsymbol {\gh a}^+, t) - O(e^{3\gamma_0 t}  (\| V(t) \|_{L^2}^2 + e^{-2 \gamma_0 t}) \| \boldsymbol a_+(t) \|_{\ell^2}),
\end{align*}
where $c_0 = \frac 12\sqrt{\lambda_0} \min_i\{1/\gamma_i\} >0$.  Note that  from \eqref{ga0} we have $2c_0 - 3 \gamma_0> \gamma_0>0$.
Now, at time $T^* = T^*(\boldsymbol {\gh a}^+)$, $\| V(T^*) \|_{L^2} = O(e^{-\gamma_0 T^*})$, whereas $\| \boldsymbol a_+(T^*) \|_{\ell^2} = e^{-3\gamma_0 T^*/2}$, i.e. $\q N(\boldsymbol{\gh a}^+, T^*)=1$ , hence
\[ 
\left. \frac{d}{dt} \q N(\boldsymbol {\gh a}^+, t) \right|_{t =T^*(\boldsymbol {\gh a}^+)} \le -(2c_0- 3 \gamma_0) + O( e^{-\gamma_0 T^*/2}).
\]
Choosing $T_0$ larger if necessary, and as $T^* \ge T_0$ for all $\boldsymbol{\gh a}_+$, we get
\begin{equation} \label{eq:a_out}  
\left. \frac{d}{dt} \q N(\boldsymbol {\gh a}_+, t) \right|_{t =T^*(\boldsymbol {\gh a}_+)}  \le -  \frac12 \gamma_0 <0.
\end{equation}
This concludes the proof of Proposition \ref{prop:flow}.
\end{proof}

We end this paragraph with the proof of Corollary 
\ref{cor:flow}.

\begin{proof}[Proof of Corollary \ref{cor:flow}]
Let us now show that $\Omega$ is open and that the mapping $\boldsymbol{\gh a}_+ \mapsto T^*(\boldsymbol{\gh a}_+)$ is continuous. Let $\boldsymbol{\gh a}_+ \in \Omega$. We recall that  $\q N(\boldsymbol{\gh a}_+, t) = e^{3\gamma_0 t} \| \boldsymbol a_+(t) \|_{\ell^2}^2$.
By \eqref{eq:a_out}, for all $\e>0$ small, there exists $\delta>0$ such that 
\begin{itemize}
\item $\q N (\boldsymbol{\gh a}_+, T^*(\boldsymbol{\gh a}_+) - \e)  > 1+\delta$, and

\smallskip

\item for all $t \in [T^*(\boldsymbol{\gh a}_+)+\e,S_0]$ (possibly empty), $\q N(\boldsymbol{\gh a}_+, t) < 1-\delta$.
\end{itemize} 
By continuity of the flow of the \eqref{NLKG} equation, it follows that there exists $\eta >0$ such that the following holds. For all $\boldsymbol{\tilde {\gh a}}_+ \in B_{\m R^N}(0,1)$ such that $\| \boldsymbol{\tilde {\gh a}}_+ - \boldsymbol{\gh a}_+ \|\le  \eta$, 
then $|\q N(\boldsymbol{\tilde {\gh a}}_+,t) - \mathcal{N}(\boldsymbol{\gh a}_+,t)| \le  \delta/2$ for all $t \in [T^*(\boldsymbol{\gh a}_+)-\e,S_0]$. In particular, $\tilde {\gh a}_+ \in \Omega$ and
\[
T^*(\boldsymbol{\gh a}_+) - \e \le T^*(\boldsymbol{\tilde {\gh a}}_+) \le T^*(\boldsymbol{\gh a}_+) + \e.
\]
 This exactly means that $\Omega$ contains a neighbourhood of ${\gh a}_+$, hence is open, and that $\boldsymbol{\gh a}_+ \mapsto T^*(\boldsymbol{\gh a}_+)$ continuous.

\medskip

Finally, let us show that the exit is instantaneous on the sphere. If $\| \boldsymbol{\gh a}_+ \|_{\ell^2} =1$, then $\q N(\boldsymbol{\gh a}_+,S_0) =1$, hence by \eqref{eq:a_out}, $\q N(\boldsymbol{\gh a}_+,t) > 1$ for all $t <S_0$ in a neighborhood of $S_0$. This means that $T^*(\boldsymbol{\gh a}_+) = S_0$.
\end{proof}

\subsection{Extension to higher dimension}

The main part of the proof remains unchanged. One has to work only for the definition of the Lyapunov functional. The key point is to notice that one can find a suitable direction as in \cite{Martel1}. The set
\[ 
\q M = \big\{ \beta \in \m R^d \mid \forall j,\ \beta \cdot \beta_j =0 \big\}, 
\]
is of zero measure: let $\bar \beta \notin \q M$; up to rescaling, we can assume $|\bar \beta| =1$. Without loss of generality we can assume that the indexes $j$ satisfy
\[
-1< (\bar \beta \cdot  \beta_1 )< (\bar \beta \cdot  \beta_2) < \cdots < (\bar \beta \cdot  \beta_N) <1.
\]
We use again the $1$d cut-off function $\phi$ defined at Step 3 of the previous
to define the new cut-off functions
\[ \psi_j (x) = \phi \left( \frac{\bar \beta \cdot x - m_j t}{L} \right)-
\phi \left( \frac{\bar \beta \cdot x - m_{j+1} t}{L} \right), \quad \text{where} \quad m_j = \frac 1 2 ( \beta_j + \beta_{j-1}) \cdot \bar \beta. \]
Then all the computations of Step 3 of Section \ref{4} follow unchanged. We refer to \cite{Martel1} (Claim 1 and what follows) for further details.

\section{Proof of Theorem \ref{MT1}} \label{sec:4}

The proof of Theorem \ref{MT1} follows from Proposition \ref{prop:app_Nsol} in a standard fashion, see e.g. \cite{Martel1}. The main point is continuity of the flow for the weak $H^1 \times L^2$ topology. 

\begin{lem}
The \eqref{NLKG} flow is continuous for the weak $H^1 \times L^2$ topology. More precisely, let $U_n \in \q C([0,T], H^1 \times L^2)$ be a sequence of solutions to \eqref{NLKG}, and assume that for some $M >0$,
\[ U_n(0) \tendf U^* \quad \text{in } H^1 \times L^2-weak, \text{ and} \quad \forall n, \quad  \| U_n(t) \|_{\q C([0,T],H^1 \times L^2)} \le M. \]
Define $U \in \q C([0,T^+(U)), H^1 \times L^2)$ be the solution to \eqref{NLKG} with initial data $U(0) = U^*$. Then $T^+(U) > T$ and
\[ \forall t \in [0,T], \quad U_n(t) \tendf U(t) \quad \text{in } H^1 \times L^2-weak. \]
\end{lem}

\begin{proof}
This is a simple consequence of the local well posedness of \eqref{NLKG} in $H^s \times \dot H^{s-1}$ for some $s <1$. More precisely, we have

\begin{thm*}[Local wellposedness]
There exists $0 \le s_{f,d} <1$ such that for all $s \ge s_{f,g}$, the following holds. Given any data $U_0 = (u_0,u_1) \in H^s \times \dot H^{s-1}$, there exist a unique solution $U \in \q C([0,T^+(U)), H^s \times \dot H^{s-1})$ to \eqref{NLKG} such that $U(0)=U_0$. Furthermore, 
\begin{enumerate}
\item  The maximal time of existence $T^+(U)$ is the same in all $H^{\sigma} \times \dot H^{\sigma-1}$ for $\sigma \in [s_{f,d},s]$. If finite, it is characterized by 
\[ \lim_{t \to T^+(U)} \| U(t) \|_{H^s \times \dot H^{s-1}} = +\infty. \]
\item The flow is continuous, in the sense that if $U_n$ is a sequence of solution to \eqref{NLKG} such that $U_n(0) \to U(0)$ in $H^s \times \dot H^{s-1}$, then  $T^+(U) \ge \liminf_n T^+(U_n)$ and
\[ \forall t \in [0,T^+(U)), \quad \| U_n - U \|_{\q C([0,t],H^s \times \dot H^{s-1})} \to 0 \quad \text{as} \quad n \to +\infty. \]
\end{enumerate}
\end{thm*}
We refer to \cite[Theorem 1.2]{Tao99} and the Remark following it for a proof and the precise value of $s_{f,d}$ (which is not important for us). 

\medskip 

Fix $0 < s<1$ be such that the Theorem holds. Let $t \in [0,T]$ such that $t < T^+(U)$. 

Let $V \in (\q D(\m R^d))^2$ and $R>0$ such that $\mathrm{Supp}\ V \subset B_{\m R^d}(0,R)$.

As $U_n(0) \tendf U(0)$ weakly in $H^1 \times L^2(\m R^d)$,  there holds by Sobolev compact embedding
\[ \| U_n(0) - U(0) \|_{H^s \times \dot H^{s-1}(B_{\m R^d}(0,R+t))} \to 0. \]
It follows by finite speed of propagation and the continuity of the flow in the local well-posedness Theorem that
\[ \| U_n(t) - U(t) \|_{H^s \times \dot H^{s-1}(B_{\m R^d}(0,R))} \to 0. \]
Hence denoting $U_n = (u_n, \partial_t u_{n})$ and $V =(v_0,v_1)$,
\begin{align*}
& | \langle U_n(t) - U(t), V \rangle | \\
& \quad = \left| \int_{|x| \le R} \left( (\partial_t u_{n}(t,x) -  \partial_t u_{n}(t,x) ) v_1(x) \vphantom{\int} \right. \right. \\
& \qquad \left. \left. \vphantom{\int} + \nabla (u_n-u) \cdot \nabla v_0 + (u_n(t,x)-u(t,x)) v_0(x) \right) dx \right| \\ 
& \quad  \le  \| U_n(t) - U(t) \|_{H^s \times \dot H^{s-1}(B_{\m R^d}(0,R))} \| V \|_{H^{2-s} \times \dot H^{1-s}} \to 0.
\end{align*}
Therefore $U_n(t) \tendf U(t)$ in $\q D'$, and by the $H^1 \times L^2$ bound, $U_n(t) \tendf U(t)$ weakly in $H^1 \times L^2$.

In particular $\| U(t) \|_{H^1 \times L^2} \le \liminf_{n \to \infty} \| U_n(t) \| \le M$. From there, a continuity argument shows that $T^+(U) > T$.
\end{proof}

We can now prove Theorem \ref{MT1}. Let $(S_n)_{n\geq 1} \subset \R$ be a sequence that satisfies $S_n > S_0$, $S_n $ increasing and $S_n \to +\infty$. 
From Proposition \ref{prop:app_Nsol} there exists a sequence of final data functions $U_{0,n} \in H^1 \times L^2$ such that
\be\label{suite}
\forall t \in [T_0, S_n],\quad  U_n (t):= \varphi(S_n,t,U_{0,n}) \in \mathcal V(t).
\ee
(We recall that $\varphi$ denotes the flow and is defined in \eqref{phi}). Note that $T_0$ does not depend on $S_n$, and observe that there exists $M$ independent of $n$ such that
\be \label{UnBound}
\forall t \in  [T_0, S_n], \quad \| U_n(t) - R(t) \|_{H^1 \times L^2} \le M e^{- \gamma_0 t}.
\ee
Let $U^*_0$ be a weak limit in $H^1 \times L^2$ of the bounded sequence $U_n(T_0)$, and define
\[ U^*(t) =\varphi (t, T_0, U_0^*). \]
Fix $t \ge T_0$. Then the previous Lemma applies on $[T_0,t]$ and shows that $T^+(U^*) > t$ and
$U_n(t) \tendf U^*(t)$ weakly in $H^1 \times L^2$. Hence \eqref{UnBound} yields
\[ \| U^*(t) - R(t) \|_{H^1 \times L^2} \le \liminf_n  \| U_n(t) - R(t) \|_{H^1 \times L^2} \le M e^{- \gamma_0 t}. \]
Therefore, $T^+(U^*) = +\infty$ and $U^*$ is the desired multi-soliton.

\bigskip

\appendix

\section{The orbit of \texorpdfstring{$Q$}{Q} under general Lorentz transformations}\label{Ap}

\medskip

In this appendix we prove that the orbit of $Q$ under the group generated by space and time translations, and general Lorentz transforms is
\[ \q F := \{ (t,x) \mapsto Q_\beta(t,x-x_0) \mid \beta, x_0 \in \m R^d, \ |\beta| <1 \}. \]
We recall that we consider $Q$ as a function of time with the slight abuse of notation $Q(t,x) =Q(x)$.

The map $\beta \mapsto \Lambda_\beta$ (see \eqref{Lor}) is a group homomorphism from $(B_{\R^d}(0,1), \oplus )$ to $(M_{d+1}(\m R), \circ)$, where $\oplus$ denotes Einstein's velocity addition
\[ x \oplus y = \frac{1}{1 + x \cdot y} \left( y + \frac{x \cdot y}{|y|^2} y + \sqrt{1-|y|^2} \left( x - \frac{x \cdot y}{|y|^2} y \right) \right). \]  
In particular, $\Lambda_{-\beta} \Lambda_\beta = \Id_{d+1}$.

\medskip

A general Lorentz transform is an element of $O(1,d) \simeq \m R^d \rtimes O(d)$, hence can be written in the form
\[ \Lambda_{U,\beta} := \begin{pmatrix}
1 &  0 \ \cdots \ 0 \\
\ \begin{matrix}
0 \\
\vdots \\
0 \end{matrix}
& U 
\end{pmatrix} \Lambda_\beta, \quad \text{where} \quad  U \in SO(d), \text{ i.e }  UU^T = \Id_d. \]

As $Q(x)$ is radially symmetric, it follows that 
\[ Q \left( \Lambda_{U,\beta} \begin{pmatrix} t \\ x \end{pmatrix} \right) = Q \left( \Lambda_\beta \begin{pmatrix} t \\ x \end{pmatrix} \right) = Q_\beta (x), \]
hence the orbit of $\{ Q \}$ under general Lorentz transform is simply  $\{ Q_\beta | \beta \in \m R^d \}$. We now want to parametrize the other invariances of \eqref{NLKG}, that is time and space shifts. Fortunately, the former reduce to the latter. Indeed, notice that
\[ \beta \beta^T \sim \begin{pmatrix}
|\beta|^2 & 0 \ \cdots \ 0 \\
\; \begin{matrix}
0 \\
\vdots \\
0 \end{matrix}
& 0
\end{pmatrix}, \quad \text{so that} \quad \Id_d + \frac{\gamma-1}{|\beta|^2} \beta \beta^T \sim 
\begin{pmatrix}
\gamma &  0 \  \cdots \ 0 \\
\ \begin{matrix}
0 \\
\vdots \\
0 \end{matrix}
& \Id_{d-1}
\end{pmatrix}.
 \]
(Here $\sim$ indicates similarity of matrices). In particular, $ \Id_d + \frac{\gamma-1}{|\beta|^2} \beta \beta^T $ is invertible. Then time translations for $Q_\beta$ can be rethought as an adequate space shift:
\begin{align*}
 Q_\beta(t+t_0,x) & = Q \left( \Lambda_\beta \begin{pmatrix} t+t_0 \\ x \end{pmatrix} \right) = Q\left( \Lambda_\beta \begin{pmatrix} t \\ x \end{pmatrix} + t_0 \begin{pmatrix} \gamma \\ -\beta \end{pmatrix} \right) \\
& = Q\left( \Lambda_\beta \begin{pmatrix} t \\ x \end{pmatrix} - t_0 \begin{pmatrix} 0  \\ \beta \end{pmatrix} \right)  = Q_\beta  \begin{pmatrix}
t \\
x - t_0 (\Id_d + \frac{\gamma-1}{\beta|^2} \beta \beta^T)^{-1}(\beta)
\end{pmatrix} . 
\end{align*}
It follows that $\q F$ is stable through all general Lorentz transform, time and space shifts, hence it is the orbit of $Q$ through the group generated by these transformations.

\bigskip
\bigskip
\bigskip
\bigskip
\bigskip
\bigskip

\centerline{\scshape Rapha\"el C\^{o}te}
\medskip
{\footnotesize
\begin{center}
CNRS and École polytechnique \\
Centre de Mathématiques Laurent Schwartz UMR 7640 \\
Route de Palaiseau, 91128 Palaiseau cedex, France \\
\email{cote@math.polytechnique.fr}
\end{center}
} 

\bigskip

\medskip

\centerline{\scshape Claudio Mu\~noz}
\medskip
{\footnotesize
 \centerline{Department of Mathematics, University of Chicago}
\centerline{5734 South University Avenue, Chicago, IL 60615, U.S.A.}
\centerline{\email{cmunoz@math.uchicago.edu}}
} 

\end{document}